\documentclass[11pt,twoside]{amsart}
\usepackage{latexsym,amssymb,amsmath}
\usepackage[all]{xy}
\usepackage{tikz,pgfplots}
\usepackage{extpfeil}
\usepackage{afterpage}
\usepackage{float}
\usepackage{hyperref}

\makeatletter  
\@namedef{subjclassname@2020}{%
\textup{2020} Mathematics Subject Classification}
\makeatother

\numberwithin{equation}{section}
\newtheorem{theorem}{Theorem}[section]
\newtheorem{lemma}[theorem]{Lemma}
\newtheorem{proposition}[theorem]{Proposition}
\newtheorem{corollary}[theorem]{Corollary}

\theoremstyle{definition}
 
\newtheorem{procedure}[theorem]{Procedure}

\newtheorem{example}[theorem]{Example}

\begin{document}

\title[Normality criteria for monomial ideals]
{Normality criteria for monomial ideals} 
\thanks{The first and third authors were supported by SNI, M\'exico. 
The second author was supported by a scholarship from CONACYT, M\'exico.}

\author[L. A. Dupont]{Luis A. Dupont}
\address{
Facultad de Matem\'aticas\\
Universidad Veracruzana\\ 
Circuito Gonzalo Aguirre Beltr\'an S/N
Zona Universitaria\\
Xalapa, Veracruz, M\'exico, CP 91090.
}
\email{ldupont@uv.mx}

\author[H. Mu\~noz-George]{Humberto Mu\~noz-George}
\address{
Departamento de
Matem\'aticas\\
Centro de Investigaci\'on y de Estudios
Avanzados del
IPN\\
Apartado Postal
14--740 \\
Ciudad de M\'exico, M\'exico, CP 07000.
}
\email{munozgeorge426@gmail.com}

\author[R. H. Villarreal]{Rafael H. Villarreal}
\address{
Departamento de
Matem\'aticas\\
Centro de Investigaci\'on y de Estudios
Avanzados del
IPN\\
Apartado Postal
14--740 \\
Ciudad de M\'exico, M\'exico, CP 07000.
}
\email{vila@math.cinvestav.mx}


\begin{abstract}
In this paper we study the normality of monomial ideals using linear
programming and graph theory. We give
normality criteria for monomial 
ideals, for ideals generated by monomials of degree two, and for edge 
ideals of graphs and clutters and their ideals of covers.  
\end{abstract}

\maketitle 

\section{Introduction}\label{section-intro}

Let $\mathcal C$ be a \textit{clutter} with vertex 
set $V(\mathcal{C})=\{t_1,\ldots,t_s\}$, that is, $\mathcal C$ is a 
family of subsets $E(\mathcal{C})$ of $V(\mathcal{C})$, called edges,
none of which is contained in
another \cite[Chapter~1]{cornu-book}. For example, a graph (no multiple edges or loops) is a
clutter.  Regarding each vertex $t_i$ as a
 variable, we consider the polynomial ring $S=K[t_1,\ldots,t_s]$ over
 a field $K$. The monomials of $S$ are denoted
by $t^a:=t_1^{a_1}\cdots t_s^{a_s}$, $a=(a_1,\dots,a_s)$ in $\mathbb{N}^s$, where
$\mathbb{N}=\{0,1,\ldots\}$. In particular $t_i=t^{e_i}$, where 
$e_i$ is the $i$-th unit vector in $\mathbb{R}^s$. The
 \textit{edge ideal} of $\mathcal{C}$, denoted $I(\mathcal{C})$, 
is the ideal of $S$ given by 
$$I(\mathcal{C}):=(\textstyle\{\prod_{t_i\in e}t_i\mid e\in
E(\mathcal{C})\}).$$
\quad The
minimal set of generators of $I(\mathcal{C})$, denoted
$\mathcal{G}(I(\mathcal{C}))$, is  the
set of all squarefree monomials $t_e:=\prod_{t_i\in e}t_i$ such 
that $e\in E(\mathcal{C})$. 
Any squarefree monomial ideal $I$ of $S$ is the edge
ideal $I(\mathcal{C})$ of a clutter $\mathcal{C}$ with vertex set
$V(\mathcal{C})=\{t_1,\ldots,t_s\}$.  

A set of vertices $C$ of $\mathcal{C}$ is called a \textit{vertex
cover} if every edge of $\mathcal{C}$ contains at least one vertex of
$C$. A  \textit{minimal vertex cover} of $\mathcal{C}$ is a vertex cover which
is minimal with respect to inclusion. The clutter of minimal vertex
cover of $\mathcal{C}$, denoted $\mathcal{C}^\vee$, is called the
\textit{blocker} of $\mathcal{C}$, and the edge ideal 
$I(\mathcal{C}^\vee)$ of $\mathcal{C}^\vee$ is called the
\textit{ideal of covers} of $\mathcal{C}$ and is denoted by
$I_c(\mathcal{C})$  \cite[p.~221]{monalg-rev}. 

Let $I$ be a monomial ideal of $S$ and let
$\mathcal{G}(I):=\{t^{v_1},\ldots,t^{v_q}\}$ be the minimal set of
generators of $I$. The \textit{incidence matrix} of $I$, denoted by $A$,
is the $s\times q$ matrix with column vectors $v_1,\ldots,v_q$. If $I$
is the edge ideal of a clutter $\mathcal{C}$, this
matrix is called the \textit{incidence matrix} of $\mathcal{C}$
\cite[p.~6]{Schr}. The
$\textit{covering polyhedron}$ of $I$, denoted by
$\mathcal{Q}(I)$, is the rational polyhedron
$$
\mathcal{Q}(I):=\{x\mid x\geq 0;\,xA\geq 1\},
$$
where $1=(1,\ldots,1)$. The map
$\mathcal{C}^\vee\rightarrow\{0,1\}^s$, $C\mapsto\sum_{t_i\in C}e_i$, 
induces a bijection between $\mathcal{C}^\vee$ and the set of integral vertices
of $\mathcal{Q}(I)$ \cite[Corollary~13.1.3]{monalg-rev}. 

The \textit{Newton
polyhedron} of a monomial ideal $I$, denoted ${\rm NP}(I)$, is the 
rational polyhedron 
\begin{equation*}
{\rm NP}(I)=\mathbb{R}_+^s+{\rm 
conv}(v_1,\ldots,v_q),
\end{equation*}
where $\mathbb{R}_+=\{\lambda\in\mathbb{R}\mid \lambda\geq
0\}$. Given an integer $n\geq 1$, the \textit{integral closure} of
 $I^n$, denoted $\overline{I^n}$, can be described as 
\begin{equation}\label{jun21-21}
\overline{I^n}=
(\{t^a\mid a/n\in{\rm NP}(I)\})=(\{t^a\mid
\langle a, u_i\rangle\geq n\mbox{ for }i=1,\ldots,p\}),
\end{equation}
where $\langle\ ,\, \rangle$ denotes the 
ordinary inner product and $u_1,\ldots,u_p$ are the vertices of
$\mathcal{Q}(I)$ \cite[Theorem~3.1, Proposition~3.5]{reesclu}. 
If $I=\overline{I}$, $I$ is said to be \textit{integrally closed}. 
If $I^n=\overline{I^n}$ for all $n\geq 1$, $I$ is said to be
\textit{normal}. The edge ideal of a clutter is integrally closed
\cite[p.~153]{monalg-rev}.

In this paper we study the normality of monomial ideals, and the normality
of edge ideals of graphs and clutters and their ideals of covers. A
main problem 
in this area is the characterization of the
normality of the ideal of covers $I_c(G)$ of a graph $G$ in terms of
the combinatorics of $G$. If $G$ is an odd cycle or a perfect graph, 
then $I_c(G)$ is normal, see \cite{Al-Ayyoub-Nasernejad-cover-ideals}
and \cite{perfect}, respectively. 

The contents of this paper are as follows. 
In Section~\ref{section-prelim}, 
we introduce a few results from
polyhedral geometry, commutative algebra, and linear programming. 

As we now recall there are two well known characterizations of the normality of
$I$, one comes from commutative algebra and uses Rees algebras, and the other comes from 
integer programming and uses Hilbert bases.   

If $R$ is an integral domain with field
of fractions $K_R$, recall that the \textit{normalization} or
\textit{integral closure} of $R$ is the
subring $\overline{R}$ consisting of all the elements of $K_R$ that
are integral over $R$. If $R=\overline{R}$ we say that $R$ is {\it
normal\/}. The \textit{Rees algebra} of $I$ is the monomial subring  
$$S[Iz]=K[t_1,\ldots,t_s,t^{v_1}z,\ldots,t^{v_q}z],$$
where 
$z=t_{s+1}$ is a new variable. The Rees algebra of $I$ can be written
as 
\begin{equation}\label{jul30-22-1}
S[Iz]=S\oplus Iz\oplus\cdots
\oplus I^nz^n\oplus\cdots\subset S[z].
\end{equation}
\quad It is well known \cite[p.~168]{Vas1} that the integral 
closure of $S[Iz]$ is given by
\begin{equation}\label{jul30-22-2}
\overline{S[Iz]}=S\oplus \overline{I}z\oplus\cdots
\oplus \overline{I^n}z^n\oplus\cdots\subset S[z].
\end{equation}
\quad Thus, the ring $S[Iz]$ is normal if and only if the ideal $I$
is normal. Normal monomial subrings arise in the theory of toric varieties
\cite{cox-toric-book}. 
A finite set $\mathcal{B}\subset\mathbb{R}^s$
is called a {\it Hilbert
basis\/} if 
$$
\mathbb{Z}^s\cap \mathbb{R}_+{\mathcal B}=\mathbb{N}{\mathcal B},
$$ 
where ${\mathbb R}_+\mathcal{B}$ is the cone generated by
$\mathcal{B}$ and $\mathbb{N}{\mathcal B}$ is the semigroup spanned
by $\mathcal{B}$.  
Following \cite{normali}, we define the \textit{Rees cone} of the
ideal $I$, denoted ${\rm
RC}(I)$, as the rational cone 
\begin{equation}\label{rees-cone-eq}
{\rm RC}(I):=\mathbb{R}_+\mathcal{A}'
\end{equation}
generated by the set
$\mathcal{A}':=\{e_1,\ldots,e_s,(v_1,1),\ldots,(v_q,1)\}$, where $e_i$
is the $i$-th unit vector in $\mathbb{R}^{s+1}$.  
The ideal $I$ is normal if and only if $\mathcal{A}'$ is a
Hilbert basis (Lemma~\ref{normal-hilbert}).

In Section~\ref{normality-criteria} we give normality criteria for
monomial ideals and membership tests to determine whether or not a
given monomial lies in the integral closure of $I^n$ or is a minimal
generator of the integral closure of $I^n$. Let $I_1$ and $I_2$ be ideals of $S$
generated by monomials in disjoint sets of variables, we show that 
$I_1I_2$ is normal if and only if $I_1$ and $I_2$ are normal
(Proposition~\ref{may19-22}).

Given a vector $c=(c_1,\ldots,c_p)$ in $\mathbb{R}^p$, we set
$|c|:=\sum_{i=1}^pc_i$ and denote the integral part of $c$ by $\lfloor
c\rfloor$ and the ceiling of $c$ by $\lceil
c\rceil$. We denote the nonnegative rational numbers by $\mathbb{Q}_+$. 

We come to our first normality criterion. 

\noindent \textbf{Proposition~\ref{normality-criterion-ip}}\textit{\ Let $I$ be a monomial
ideal of $S$ and let $A$ be its incidence matrix. The following conditions are equivalent.
\begin{enumerate}
\item[(a)] $I$ is a normal ideal.
\item[(b)] For each pair of vectors $\alpha\in\mathbb{N}^s$ and
$\lambda\in\mathbb{Q}_+^q$ such that $A\lambda\leq\alpha$, there is 
$m\in\mathbb{N}^q$ satisfying $Am\leq\alpha$ and
$|\lambda|=|m|+\epsilon$ with $0\leq\epsilon<1$.
\end{enumerate} 
}

Given a monomial ideal $I$ and a monomial $t^\alpha$, a 
linear programming membership test for the question
``is $t^\alpha$ a member of $\overline{I}$?'' was shown 
in \cite[Proposition~3.5]{mc8}, \cite[Proposition~1.1]{Ha-Trung-19}. The following proposition gives
a linear algebra membership test that complement these results. 

\noindent \textbf{Proposition~\ref{membership-test-n}\ }(Membership test)\textit{
Let $I$ be a monomial 
ideal of $S$, let $A$ be its incidence matrix, and let $t^\alpha$ be a
monomial in $S$. The following are equivalent.
\begin{enumerate}
\item[(a)] $t^\alpha\in\overline{I^n}$, $n\geq 1$.
\item[(b)] $A\lambda\leq\alpha$ for some $\lambda\in\mathbb{Q}_+^q$
with $|\lambda|\geq n$. 
\item[(c)] $\max\{\langle y, 1\rangle\mid 
y\geq 0;\, Ay\leq\alpha \}=\min\{\langle\alpha,x\rangle\mid x\geq 0;\,
xA\geq
1\}\geq n$.
\end{enumerate}
}

As a byproduct we obtain a minimal generators test for the integral
closure of the powers of a monomial ideal.

\noindent \textbf{Proposition~\ref{mg-test}}\textit{
Let $I$ be a monomial
ideal of $S$ and let $A$ be its incidence matrix. A monomial $t^\alpha\in
S$ is a minimal generator of
$\overline{I^n}$ if and only if the following two conditions hold.
\begin{align*}
&\max\{\langle y, 1\rangle\mid 
y\geq 0;\, Ay\leq\alpha \}=\min\{\langle\alpha,x\rangle\mid x\geq 0;\,
xA\geq
1\}\geq n.\\
&\max\{\langle y, 1\rangle\mid 
y\geq 0;\, Ay\leq\alpha-e_i\}=\min\{\langle\alpha-e_i,x\rangle\mid x\geq 0;\,
xA\geq
1\}<n\\
&\mbox{for each }e_i\mbox{ for which }\alpha-e_i\geq 0.\nonumber 
\end{align*}}
\quad The equality $\overline{I^n}=\overline{(t^{nv_1},\ldots,t^{nv_q})}$
for $n\geq 1$ comes from \cite[Examples~1.4, 3.7]{mc8}. We use the membership
test to give a short proof of this equality (Corollary~\ref{n-th-power}).

The normality of $I$ is also related to integer rounding
properties \cite[Corollary~2.5]{poset}. 
The linear system $x\geq 0; xA\geq{1}$ 
has the {\it integer rounding property\/} if 
\begin{equation}\label{irp-eq1}
{\rm max}\{\langle y,{1}\rangle\mid y\in\mathbb{N}^q;\,   
Ay\leq \alpha\} 
=\lfloor{\rm max}\{\langle y,{1}\rangle\mid y\geq 0;\, 
Ay\leq \alpha\}\rfloor
\end{equation}
for each integral vector $\alpha$ for which the 
right-hand side is finite. The linear system $x\geq 0; xA\leq{1}$ 
has the \textit{integer rounding property} if 
\begin{equation}\label{irp-eq2}
\lceil{\rm min}\{\langle y,{1}\rangle\mid y\geq 0;\, Ay\geq \alpha \}\rceil
={\rm min}\{\langle y,{1}\rangle\mid y\in\mathbb{N}^q;\, Ay\geq \alpha\}
\end{equation}
for each integral vector $\alpha$ for which the left hand side is
finite. Systems with the integer rounding property are well studied;
see \cite{baum-trotter,ainv}, \cite[Chapter~22]{Schr},
\cite[Chapter~5]{Schr2}, 
and references therein. 

As an application we give a short 
proof of the fact that $I$ is normal if and only if 
the system $x\geq 0;xA\geq{1}$ has the integer rounding
property \cite[Corollary~2.5]{poset}
(Corollary~\ref{crit-rounding-normali}). This fact was shown in \cite{poset} using the theory of
blocking and antiblocking polyhedra \cite{baum-trotter},
\cite[p.~82]{Schr2}. 

The following notions of {\it contraction}, {\it deletion}, and {\it
minor\/} come from combinatorial optimization \cite{Schr2}. Given a
clutter $\mathcal{C}$ and a vertex $t_i\in
V(\mathcal{C})$, the {\it
contraction} 
$\mathcal{C}/t_i$ and {\it deletion} $\mathcal{C}\setminus{t_i}$ are the
clutters constructed as follows: both have $V(\mathcal{C})\setminus\{t_i\}$ as
vertex set, $E(\mathcal{C}/t_i)$ is the set of minimal elements 
of $\{e\setminus\{t_i\}\vert\, e\in E(\mathcal{C}\}$, minimal with
respect to inclusion, and 
$E(\mathcal{C}\setminus{t_i})$ is the set 
$\{e\vert\, t_i\notin e\in E(\mathcal{C})\}$. A {\it minor} of $\mathcal{C}$
is a clutter obtained from $\mathcal{C}$ by a sequence of deletions 
and contractions in any order. 

If $\mathcal{C}$ is a clutter and $I(\mathcal{C})$ is a normal ideal, then 
$I(\mathcal{H})$ is a normal ideal for any minor $\mathcal{H}$ of
$\mathcal{C}$ \cite[Proposition~4.3]{normali}. The following result
complements this fact.

\noindent \textbf{Proposition~\ref{may21-22}}\textit{\ Let $\mathcal{C}$ be a clutter and let
$I_c(\mathcal{C})$ be its ideal of covers. If $I_c(\mathcal{C})$ is
normal, then $I_c(\mathcal{H})$ is normal for any minor $\mathcal{H}$ of $\mathcal{C}$.
}

In Section~\ref{normality-degree2} we use Hilbert basis, Ehrhart
rings, and a duality for integer rounding properties, to examine the
normality of ideals generated by monomials of degree $2$, and generalize some results
that were previously known to be valid for edge ideals of graphs. 

In what follows
$I=(t^{v_1},\ldots,t^{v_q})$ is an ideal 
generated by monomials of degree $2$ and $A$ is the incidence matrix
of $I$. We set 
\begin{align}
&\mathcal{B}:=\{e_{s+1}\}
\cup\{e_i+e_{s+1}\}_{i=1}^s\cup\{(v_i,1)\}_{i=1}^q,\label{jul20-22}\\
&\mathcal{Q}:={\rm conv}(0,e_1,\ldots,e_s,v_1,\ldots,v_q),\\
&R:=K[\mathbb{N}\mathcal{B}]=K[z,t_1z,\ldots,t_sz,t^{v_1}z,\ldots,t^{v_q}z],
\mbox{ the \textit{semigroup ring} of }\mathbb{N}\mathcal{B},\\
&{\rm Er}(\mathcal{Q}):=K[t^{\alpha}z^b\mid \alpha\in\mathbb{Z}^s\cap
b\mathcal{Q}]\subset
S[z], \mbox{ the \textit{Ehrhart ring} of }\mathcal{Q}.\label{jul20-22-4}
\end{align}

The following is one of our main results.

\noindent \textbf{Theorem~\ref{normal-b}}\textit{\ Let
$I=(t^{v_1},\ldots,t^{v_q})$ 
be an ideal of $S$ generated by monomials
of degree $2$. Then, $I$ is a normal ideal if and only if
$\mathcal{B}$ is a Hilbert basis.
}

As an application we recover the fact that if $I$ is the edge ideal
of a connected graph, then $I$ is normal if and only if 
$K[\mathbb{N}\mathcal{B}]$ is normal \cite[Theorem~3.3]{roundp}
(Corollary~\ref{jul28-22}).

We characterize when the Ehrhart ring of $\mathcal{Q}$ is the monomial
subring $K[\mathbb{N}\mathcal{B}]$ using the integer rounding
property.

\noindent \textbf{Theorem~\ref{ehrhart-q}}\textit{ 
Let $I$ be an ideal 
of $S$ generated by monomials of degree $2$ and
let $A$ be the incidence matrix of $I$. Then, 
$\overline{K[\mathbb{N}\mathcal{B}]}={\rm Er}(\mathcal{Q})$, and the equality 
$$ 
K[\mathbb{N}\mathcal{B}]={\rm Er}(\mathcal{Q})
$$
holds if and only if the 
system $x\geq 0;\, xA\leq 1$ has the integer
rounding property.
}

In particular, we recover the fact that if $I$ 
is the edge ideal
of a connected graph, then the semigroup ring $K[\mathbb{N}\mathcal{B}]$ is normal if
and only if the system $x\geq 0; xA\leq 1$ has the integer rounding
property \cite[Theorem~3.3]{roundp}.

The \textit{dual} of the edge ideal $I$ of a clutter $\mathcal{C}$, denoted $I^*$, is the
ideal of $S$ 
generated by all monomials $t_1\cdots t_s/t_e$ such that $e$ is an
edge of $\mathcal{C}$. If $A$ and $A^*$ are the incidence matrices of
$I$ and $I^*$, respectively, then the system $x\geq 0;xA\geq{1}$
has the integer rounding property if and 
only if the system $x\geq 0;xA^*\leq{1}$ has the integer
rounding property \cite[Theorem~2.11]{ainv}
(Theorem~\ref{duality-irp}). We will use this duality to characterize
the normality of the dual of the edge ideal of a
graph using Hilbert bases.

If $I$ is the edge ideal of a graph, we prove that $I^*$ is normal if and only if
the set $\mathcal{B}$ is a Hilbert basis (Corollary~\ref{jul4-22}),
and furthermore we prove that $I^*$ is normal if and only if $I$ is
normal (Proposition~\ref{duality-i-i*}). These two results follow from \cite[Theorem~2.12]{ainv} and
\cite[Theorem~3.3]{roundp} when $I$ is the edge ideal of a connected
graph.

In Section~\ref{normality-icdual}, we study the normality of the ideal
of covers $I_c(G)$ of a graph $G$ and give a combinatorial criterion
in terms of Hochster configurations for the normality of $I_c(G)$
when the independence number of $G$ is at most 
two.

Let $v$ be a vertex of a graph $G$. Recall that if $I_c(G)$ is normal, then $I_c(G\setminus v)$ is
normal. The following results shows that the
converse holds under a certain condition. 

\noindent \textbf{Theorem~\ref{may14-22}}\textit{\ Let $v$ be a vertex of a graph $G$. If
the neighbor set $N_G(v)$ of $v$ is a minimal vertex cover of $G$,
then $I_c(G\setminus v)$ is normal if and only if $I_c(G)$ is normal.
}

As a consequence, we recover a result of Al-Ayyoub, Nasernejad and
Roberts showing that the ideal of covers of the
cone over the graph $G$ is normal if the ideal of covers of the
graph $G$ is normal
\cite[Theorem~1.6]{Al-Ayyoub-Nasernejad-cover-ideals}
(Corollary~\ref{aug19-22}).   

Let $G$ be a graph and let $G_1,\ldots,G_r$ be its connected
components. If the edge ideal $I(G)$ is normal, then the edge ideal $I(G_i)$ is normal for
$i=1,\ldots,r$ \cite[Proposition~4.3]{normali} but the converse is not
true (Example \ref{two-triangles-example}). Using Proposition~\ref{may19-22}, we show
that $I_c(G)$ is normal if and only if $I_c(G_i)$ is normal 
for $i=1,\ldots,r$ (Corollary~\ref{may20-22}). 

Hochster gave an example 
of a connected graph whose edge ideal is not normal
\cite[p.~457]{monalg-rev} (cf. \cite[Example~4.9]{ITG}). This example leads to the following concept
\cite[Definition~6.7]{ITG}. 

A \textit{Hochster configuration} of a graph $G$ consists of two odd
cycles $C\sb{1}$, $C\sb{2}$ of $G$ satisfying the following
two conditions:
\begin{enumerate}
\item[(i)] 
$C\sb {1}\cap N_G(C\sb{2})=\emptyset$, where $N_G(C_2)$
is the neighbor set of $C_2$.  
\item[(ii)] No chord of $C\sb {i}$, $i=1,2$, is
an edge of $G$, i.e., $C_i$ is an induced cycle of $G$.
\end{enumerate}

It was conjectured in \cite[Conjecture~6.9]{ITG} that the 
edge ideal of a graph $G$ is normal if and only if the graph has no Hochster configurations. This
conjecture was proved in \cite[Corollary~5.8.10]{graphs}, 
\cite[Corollary~10.5.9]{monalg-rev}. We give
a direct proof of this conjecture using Vasconcelos's description of the integral closure
of the Rees algebra of $I(G)$ \cite[p.~265]{graphs}
(Theorem~\ref{apr16-02}).

The complement of
$G$ is denoted by $\overline{G}$. Using Proposition~\ref{duality-i-i*}
and Theorem~\ref{apr16-02}, we prove that if $\overline{G}$ is 
the disjoint 
union of two odd cycles of
length at least $5$, then $I_c(G)$ is not normal, and if $\overline{G}$ is an odd cycle of
length at least $5$, then $I_c(G)$ is normal
(Corollary~\ref{normality-nc0}). Furthermore, we show that if
$I_c(G)$ is normal, 
then $\overline{G}$ has no
Hochster configuration with induced odd cycles $C_1$, $C_2$ of
length at least $5$ (Corollary~\ref{normality-nc}). 

Let $G$ be a graph with vertex set $V(G)=\{t_1,\ldots,t_s\}$.
A subset $B$ of $V(G)$ is called \textit{independent} or
\textit{stable} if $e\not\subset B$ for any  
$e\in E(G)$. The \textit{independence
number} of $G$, denoted $\beta_0(G)$, is the number 
of vertices in any largest stable set of vertices. 

The following result gives a combinatorial description of 
the normality of the ideal of covers of graphs 
with independence number at most two.

\noindent \textbf{Theorem~\ref{dim2-comb-char}}\ (Duality
criterion)\textit{ Let $G$ be a graph with $\beta_0(G)\leq 2$. The
following hold.
\begin{enumerate}
\item[(a)] $I_c(G)$ is normal if and only if $I(\overline{G})$ is normal. 
\item[(b)] $I_c(G)$ is normal if and only if $\overline{G}$ has no
Hochster configurations.
\end{enumerate}
}
\quad In Section~\ref{examples-section} and Appendix~\ref{Appendix}, we
show some examples and procedures for \textit{Normaliz}
\cite{normaliz2} and \textit{Macaulay}$2$ \cite{mac2}. 
We give an example showing that Theorem~\ref{normal-b} does 
not extend to ideals generated by monomials of degree greater than
$2$ (Example~\ref{example1}). If $\overline{G}$ is a cycle of length
$7$, then the ideals $I(G)$, $I(\overline{G})$, $I_c(G)$, and $I_c(\overline{G})$,
are all normal (Example~\ref{example2}). In
Examples~\ref{example3}--\ref{example5}, we illustrate an auxiliary
result (Lemma~\ref{jul11-22}) and the duality criterion (Theorem~\ref{dim2-comb-char}). 

The graph $G$ in Example~\ref{kaiser-example6} was introduced by Kaiser,
Stehl\'\i k, and  \v{S}krekovski \cite{persistence-ce}. They show 
that the ideal of covers of $G$ satisfies 
$${\rm depth}(S/I_c(G)^3)=0<
4={\rm depth}(S/I_c(G)^4),$$
i.e., the depth function of the cover
ideal of $G$ is not non-increasing, answering a question of Herzog
and Hibi. They also show that ${\rm Ass}(S/I_c(G)^3)$ is not a subset of
${\rm Ass}(S/I_c(G)^4)$, i.e., the ideal of covers $I_c(G)$ of $G$
does not satisfy the
persistence property of associated primes (see
\cite{Fco-Ha-VT,ass-powers} 
and references therein). The graph $G$ is denoted by $H_4$ in
\cite{persistence-ce}. 
Using the normality test of
Procedure~\ref{normality-test-procedure} and {\it Macaulay\/}$2$ \cite{mac2}, we
obtain that $I(G)$ and $I_c(G)$ are not normal whereas $I(\overline{G})$
and $I_c(\overline{G})$ are normal. This example shows that none of
the implications of the duality criterion of 
Theorem~\ref{dim2-comb-char}(a) hold for graphs with independence 
number greater than $2$.

Let $G$ be the graph of Example~\ref{example7}. This graph has
independence number equal to $3$. Using the normality
test of Procedure~\ref{normality-test-procedure} and 
\textit{Macaulay}$2$ \cite{mac2}, we obtain that $I(G)$ is normal, $I_c(G)$ is not 
normal, and furthermore $I_c(G)^5$ is not integrally closed. 
This example shows that the Hochster configurations of
$\overline{G}$, with $C_1$,
$C_2$ induced odd cycles of length at least five, are not the only obstructions to
the normality of $I_c(G)$ (see Corollary~\ref{normality-nc}).  

For unexplained
terminology and additional information,  we refer to 
\cite{mc8,huneke-swanson-book,Vas1,bookthree} for the theory of
integral closure, \cite{Herzog-Hibi-book,monalg-rev} for the theory of
edge ideals and monomial ideals, and \cite{hemmecke,Schr,Schr2} for 
the theory of Hilbert bases and polyhedral geometry.

\section{Preliminaries}\label{section-prelim} 
In this section we introduce a few results from
polyhedral geometry, commutative algebra, and linear programming. We continue to employ 
the notations and definitions used in Section~\ref{section-intro}. 

Given $a\in {\mathbb R}^s\setminus\{0\}$  and 
$c\in {\mathbb R}$, the \textit{closed halfspace} $H^+{(a,c)}$ is defined as
\[
H^+{(a,c)}:=\{x\in{\mathbb R}^s\vert\, 
\langle x,a\rangle\geq c\}.
\]
\quad If $a$ and $c$ are rational, $H^+{(a,c)}$ 
is called a \textit{rational closed halfspace}. A \textit{rational polyhedron}
is a subset of ${\mathbb R}^s$ which is the 
intersection of a finite number of rational closed halfspaces of 
$\mathbb{R}^s$. 

The \textit{cone} generated by a finite subset $\Gamma$ of $\mathbb{R}^s$, 
denoted ${\mathbb R}_+\Gamma$, is the set of all 
linear combinations of $\Gamma$ with coefficients in $\mathbb{R}_+$.
The \textit{finite basis theorem} asserts that a subset $\mathcal{Q}$ of $\mathbb{R}^s$ is a 
rational polyhedron if and only if
$$\mathcal{Q}=\mathbb{R}_+\Gamma+\mathcal{P},$$  
where $\Gamma$ is a finite set of rational points and $\mathcal{P}$
is the convex 
hull ${\rm conv}(\mathcal{A})$ of a
finite set $\mathcal{A}$ of rational
points \cite[Corollary~7.1b]{Schr}. In particular the Newton
polyhedron of a monomial ideal is a rational polyhedron.

Consider a cone $\mathbb{R}_+\Gamma\subset\mathbb{R}^s$ containing no
lines and generated by a finite set $\Gamma$ of rational points. 
A finite set ${\mathcal B}$ is called a 
{\it Hilbert basis\/} of $\mathbb{R}_+\Gamma$ if
$\mathbb{R}_+\Gamma=\mathbb{R}_+{\mathcal B}$ and  
$\mathcal{B}$ is a Hilbert basis. If $\mathcal{B}$ 
is minimal (with respect to inclusion), then 
$\mathcal{B}$ is unique \cite{Schr1}.
The program \textit{Normaliz}
\cite{normaliz2} can be used to compute the Hilbert basis of
$\mathbb{R}_+\Gamma$.  

Let $F=\{t^{w_1},\ldots,t^{w_r}\}$ be a set of monomials of $S$ and
let $\mathcal{F}$ be the family of all subrings $D$ of $S$ such that
$K\cup F\subset D$. The \textit{monomial subring} generated $F$ is
given by
$$
K[F]:=\bigcap_{D\in\mathcal F} D.
$$
\quad The elements of $K[F]$ have the form $\sum
c_a\left(t^{w_1}\right)^{a_1}\cdots\left(t^{w_r}\right)^{a_r}$  
with $c_a\in K$ and all but a finite number of $c_a$'s are zero. Let
$\mathcal{A}$ be the set of vectors
$\{w_1,\ldots,w_r\}\subset\mathbb{N}^s$.   
As a $K$-vector space $K[F]$ is generated by the
set of monomials of the form $t^a$, 
with $a$ in the semigroup $\mathbb{N}{\mathcal A}$ generated by
$\mathcal{A}$. That is, $K[F]=K[\{t^a\mid a\in\mathbb{N}{\mathcal
A}\}]$. This means that $K[F]$ coincides with
$K[\mathbb{N}{\mathcal A}]$, the {\it semigroup ring\/} of the
semigroup $\mathbb{N}{\mathcal A}$ (see \cite{gilmer}).

\begin{theorem}{\rm(\cite[pp.~29--30 ]{Ful1},
\cite[Theorem~9.1.1]{monalg-rev})}\label{int-clos-gen} If
$\Gamma\subset\mathbb{N}^s$ is a finite set of points   
and $\mathcal{S}={\mathbb Z}\Gamma\cap {\mathbb R}_+\Gamma$, then the 
following hold{\rm:}
\begin{align*}
{\rm(a)}\ K[\mathcal{S}]:=K[\{t^a\mid a\in\mathcal{S}\}] \mbox{ is
normal};&\quad {\rm(b)}\ \overline{K[F]}=K[\mathcal{S}], \mbox{ where
}F=\{t^{a}\mid a\in\Gamma\}. 
\end{align*}
\end{theorem}

\begin{lemma}\label{normal-hilbert} Let $I=(t^{v_1},\ldots,t^{v_q})$ be a monomial ideal of
$S$ and let $\mathcal{A}'$ be the subset of $\mathbb{R}^{s+1}$ given
by $\{e_i\}_{i=1}^s\cup\{(v_i,1)\}_{i=1}^q$. Then, $I$ is normal if
and only if $\mathbb{R}_+\mathcal{A}'\cap\mathbb{Z}^{s+1}=\mathbb{N}\mathcal{A}'$. 
\end{lemma}

\begin{proof} The Rees algebra of $I$ is the monomial subring 
$S[Iz]=K[t_1,\ldots,t_s,t^{v_1}z,\ldots,t^{v_q}z]$, where 
$z=t_{s+1}$ is a new variable. By
Eqs.~\eqref{jul30-22-1}--\eqref{jul30-22-2}, 
the ring $S[Iz]$ is normal if and only if the ideal $I$
is normal. On the other hand the Rees
algebra of $I$ can also be written as
\begin{equation}\label{aug6-22-1}
S[It]=K[\{t^az^b\vert\, (a,b)\in\mathbb{N}{\mathcal
A}'\}]=K[\mathbb{N}\mathcal{A}'],
\end{equation}
and since $\mathbb{Z}\mathcal{A}'=\mathbb{Z}^{s+1}$, by
Theorem~\ref{int-clos-gen}, 
one has
\begin{equation}\label{aug6-22-2}
\overline{S[It]}=K[\{t^az^b\vert\, (a,b)\in
\mathbb{Z}^{s+1}\cap \mathbb{R}_+{\mathcal
A}'\}]=K[\mathbb{Z}^{s+1}\cap\mathbb{R}_+{\mathcal A}']. 
\end{equation}
\quad Hence, by Eqs.~\eqref{aug6-22-1}--\eqref{aug6-22-2}, $S[It]$ is normal if and only if 
$\mathbb{N}{\mathcal A}'=\mathbb{Z}^{s+1}\cap\mathbb{R}_+{\mathcal A}'$.
\end{proof}

\begin{theorem}{\rm(Farkas's Lemma \cite[Theorem~1.1.25]{monalg-rev})}\label{farkas} 
Let $A$ be an $s\times q$ matrix with entries in a field $\mathbb{K}$
and let $\alpha\in \mathbb{K}^s$. 
Assume ${\mathbb Q}\subset\mathbb{K}\subset{\mathbb R}$. Then, either 
there exists $x\in\mathbb{K}^q$ with $Ax=\alpha$ and $x\geq 0$, or
there exists $\beta\in \mathbb{K}^s$ with $\beta A\geq 0$ and $\langle
\beta,\alpha\rangle<0$, but
not both. 
\end{theorem}

\begin{lemma}\cite[p.~169]{Vas1}\label{icd}
If $I$ is a monomial ideal of $S$ and $n\in\mathbb{N}_+$, then 
\begin{equation*}
\overline{I^n}=(\{t^a\in S\mid (t^a)^{p}\in I^{pn}
\mbox{ for some }p\geq 1\}).
\end{equation*}
\end{lemma}

\begin{proof} This follows using Eq.~\eqref{jun21-21} and Farkas's
lemma (Theorem~\ref{farkas}). 
\end{proof}

The following duality is valid for incidence matrices of clutters.

\begin{theorem}\cite[Theorem~2.11]{ainv}\label{duality-irp} Let $A=(a_{i,j})$ be a
$\{0,1\}$-matrix and let $A^*=(a_{i,j}^*)$ be the
matrix whose $(i,j)$-entry is $a_{i,j}^*=1-a_{i,j}$. Then, the system 
$x\geq 0;xA\geq{1}$
has the integer rounding property if and 
only if the system $x\geq 0;xA^*\leq{1}$ has the integer
rounding property.
\end{theorem}

\section{Normality criteria for monomial
ideals}\label{normality-criteria}
In this section we give normality criteria for
monomial ideals and membership tests for the integral closure of their
powers. To avoid repetitions, we continue to employ 
the notations and definitions used in Sections~\ref{section-intro} 
and \ref{section-prelim}.

\begin{proposition}\label{may19-22} Let $I_1$ and $I_2$ be ideals of $S$
generated by monomials in disjoint sets of variables. The following
hold.
\begin{enumerate}
\item[(a)] $I_1I_2=I_1\cap I_2$.\quad {\rm (b)} 
$\overline{I_1}\ \overline{I_2}\subset \overline{I_1I_2}$.  
\item[(c)] $\overline{(I_1I_2)^n}=\overline{I_1^n}\ \overline{I_2^n}$
for all $n\geq 1$.
\item[(d)] $I_1I_2$ is normal if and only if $I_1$ and $I_2$ are
normal.
\end{enumerate}
\end{proposition}

\begin{proof} (a) Clearly $I_1I_2\subset I_1\cap I_2$. To show the
reverse inclusion take $t^a\in I_1\cap I_2$. Then, we can write $t^a=t^bt^\gamma$ and
$t^a=t^ct^\delta$ with $t^b\in\mathcal{G}(I_1)$ and
$t^c\in\mathcal{G}(I_2)$. Since the monomials in $\mathcal{G}(I_1)$ 
do not have any common variable with the monomials in 
$\mathcal{G}(I_2)$, from the equality $t^bt^\gamma=t^ct^\delta$ it follows that $t^a\in I_1I_2$.

(b) Take $t^at^b\in\overline{I_1}\ \overline{I_2}$ with 
$t^a\in\mathcal{G}(\overline{I_1})$ and
$t^b\in\mathcal{G}(\overline{I_2})$. Then, by Lemma~\ref{icd}, we get
that 
$(t^a)^{k}\in I_1^k$ and $(t^b)^{\ell}\in I_2^\ell$ for 
some positive integers $k,\ell$. Thus
$$(t^at^b)^{k\ell}=(t^a)^{k\ell}(t^b)^{\ell k}\in
I_1^{k\ell}I_2^{\ell k}=(I_1I_2)^{k\ell},
$$
and consequently $(t^at^b)^{k\ell}\in(I_1I_2)^{k\ell}$. Hence 
$t^at^b\in\overline{I_1I_2}$.

(c) It suffices to show the case $n=1$ because 
$(I_1I_2)^n=I_1^nI_2^n$ for all $n\geq 1$. By parts (a) and (b), 
one has  $\overline{I_1I_2}\subset\overline{I_1}\cap \overline{I_2}=
\overline{I_1}\ \overline{I_2}\subset \overline{I_1I_2}$, and equality
holds everywhere. 

(d) $\Rightarrow)$ Assume that $I_1I_2$ is normal. To show that $I_1$
is normal we need only show the inclusion $\overline{I_1^n}\subset
I_1^n$ for all $n\geq 1$. Take $t^a\in\mathcal{G}(\overline{I_1^n})$ and fix
$t^b\in\mathcal{G}(\overline{I_2^n})$. From (c), one has
$$  
(I_1I_2)^n=\overline{(I_1I_2)^n}=\overline{I_1^n}\ \overline{I_2^n},
$$
and consequently $t^at^b\in(I_1I_2)^n$. Thus, we can write 
\begin{equation}\label{aug6-22-3}
t^at^b=(t^{c_1}t^{d_1})\cdots(t^{c_n}t^{d_n})t^\epsilon,
\end{equation}
with $t^{c_i}\in\mathcal{G}(I_1)$ and $t^{d_i}\in\mathcal{G}(I_2)$ for
$i=1,\ldots,n$. Since the monomials in $\mathcal{G}(I_1)$ 
do not have any common variable with the monomials in 
$\mathcal{G}(\overline{I_2^n})$, from Eq.~\eqref{aug6-22-3} it follows that $t^a$ is a multiple
of $t^{c_1}\cdots t^{c_n}$, and $t^a\in I_1^n$. By a similar
argument we obtain that $I_2$ is normal.

$\Leftarrow)$ Assume that $I_1$ and $I_2$ are normal. Then, 
by part (c), $I_1I_2$ is normal.
\end{proof}

\begin{proposition}\label{normality-criterion-ip} Let $I$ be a monomial
ideal of $S$ and let $A$ be its incidence matrix. The following conditions are equivalent.
\begin{enumerate}
\item[(a)] $I$ is a normal ideal.
\item[(b)] For each pair of vectors $\alpha\in\mathbb{N}^s$ and
$\lambda\in\mathbb{Q}_+^q$ such that $A\lambda\leq\alpha$, there is 
$m\in\mathbb{N}^q$ satisfying $Am\leq\alpha$ and
$|\lambda|=|m|+\epsilon$ with $0\leq\epsilon<1$.
\end{enumerate} 
\end{proposition}

\begin{proof} (a)$\Rightarrow$(b): Pick
$r\in\mathbb{N}_+=\mathbb{N}\setminus\{0\}$ such that
$A(r\lambda)\leq r\alpha$ and $r\lambda\in\mathbb{N}^q$. Let
$\lambda_1,\ldots,\lambda_q$ be the entries of $\lambda$. Then,
regarding the $v_i$'s as column vectors, one has
$$
A(r\lambda)=(r\lambda_1)v_1+\cdots+(r\lambda_q)v_q\leq
r\alpha\quad\therefore\quad(t^{\alpha})^r\in I^{r|\lambda|}\subset
I^{r\lfloor|\lambda|\rfloor}.
$$
\quad Thus, by Lemma~\ref{icd}, 
$t^{\alpha}\in\overline{I^{\lfloor|\lambda|\rfloor}}=I^{\lfloor|\lambda|\rfloor}$
because $I$ is a normal ideal. Then, we can write
$t^\alpha=(t^{v_1})^{m_1}\cdots(t^{v_q})^{m_q}t^\delta$, with
$m_i\in\mathbb{N}$ for all $i$,
$\sum_{i=1}^q{m_i}=\lfloor|\lambda|\rfloor$, and
$\delta\in\mathbb{N}^s$. If $m$ is the vector with entries
$m_1,\ldots,m_q$, then $Am\leq\alpha$ and
$|\lambda|=\lfloor|\lambda|\rfloor+\epsilon=|m|+\epsilon$ with $0\leq\epsilon<1$.

(b)$\Rightarrow$(a): To show that $I$ is a normal ideal take 
$t^\alpha\in\overline{I^n}$. Then, by
Lemma~\ref{icd}, 
$(t^\alpha)^r\in I^{rn}$ for some $r\in\mathbb{N}_+$ and we can write
$$
t^{r\alpha}=(t^{v_1})^{n_1}\cdots(t^{v_q})^{n_q}t^\delta,
$$
with $n_i\in\mathbb{N}$ for all $i$,
$\sum_{i=1}^q{n_i}=nr$, and
$\delta\in\mathbb{N}^s$. Setting $\lambda=(n_1,\ldots,n_q)/r$,
one has 
$$
\alpha=\bigg(\sum_{i=1}^q(n_i/r)v_i)\bigg)+(\delta/r)=A\lambda+(\delta/r)\geq A\lambda.
$$
\quad Hence, by hypothesis, there is 
$m\in\mathbb{N}^q$ satisfying $Am\leq\alpha$ and
$|\lambda|=|m|+\epsilon$ with $0\leq\epsilon<1$. Note that
$|\lambda|=n$, and consequently $\epsilon=0$ because $n-|m|=\epsilon$
is an integer. Thus, $n=|\lambda|=|m|$ and from the
inequality $Am\leq\alpha$ it follows readily that $t^\alpha\in I^n$. 
\end{proof}

\begin{proposition}{\rm(Membership test)}\label{membership-test-n}
Let $I=(t^{v_1},\ldots,t^{v_q})$ be a monomial
ideal of $S$, let $A$ be its incidence matrix, and let $t^\alpha$ be a
monomial in $S$. The following are equivalent.
\begin{enumerate}
\item[(a)] $t^\alpha\in\overline{I^n}$, $n\geq 1$.
\item[(b)] $A\lambda\leq\alpha$ for some $\lambda\in\mathbb{Q}_+^q$
with $|\lambda|\geq n$. 
\item[(c)] $\max\{\langle y, 1\rangle\mid 
y\geq 0;\, Ay\leq\alpha \}=\min\{\langle\alpha,x\rangle\mid x\geq 0;\,
xA\geq
1\}\geq n$.
\end{enumerate}
\end{proposition}

\begin{proof} The sets $A_1=\{y\mid y\geq 0;\, Ay\leq\alpha\}$ and
$A_2=\{x\mid x\geq 0;\,xA\geq 1\}$ are not empty because $0\in A_1$
and, since $v_i\in\mathbb{N}^s\setminus\{0\}$ for all $i$, one 
can choose the entries of $x$ large enough so that $x\in A_2$, and
furthermore the maximum in the left hand side of part (c) is finite. 
Hence, by linear programming duality \cite[Corollary~7.1g, p.~91, Eq.~(19)]{Schr}, the
equality in part (c) holds.

(a)$\Rightarrow$(b): By Lemma~\ref{icd}, $(t^\alpha)^r\in I^{rn}$ for 
some $r\in\mathbb{N}_+$. Hence, 
$$
r\alpha=p_1v_1+\cdots+p_qv_q+\delta=Ap+\delta\geq Ap,
$$
where $p=(p_1,\ldots,p_q)\in\mathbb{N}^q$, $|p|=rn$ and
$\delta\in\mathbb{N}^s$. Therefore, making $\lambda=p/r$, one obtains
that $\lambda\in\mathbb{Q}_+^q$, $A\lambda\leq\alpha$ and $|\lambda|=n$. 

(b)$\Rightarrow$(c): This is clear because $\max\{\langle y, 1\rangle\mid 
y\geq 0;\, Ay\leq\alpha \}\geq n$ and, as noted above, the equality of part
(c) holds. 

(c)$\Rightarrow$(a):  Pick an optimal feasible solution
$\lambda=(\lambda_1,\ldots,\lambda_q)\in\mathbb{Q}_+^q$ where the
maximum in the linear programming duality equation is attained. 
Then, $|\lambda|\geq n$ and $A\lambda\leq\alpha$. Choose
$r\in\mathbb{N}_+$ such that $r\lambda\in\mathbb{N}^q$. Hence
$$
A(r\lambda)=(r\lambda_1)v_1+\cdots+(r\lambda_q)v_q\leq r\alpha,
$$
and consequently $(t^{\alpha})^r\in I^{r|\lambda|}\subset
I^{rn}$. Thus, by Lemma~\ref{icd}, $t^\alpha\in\overline{I^n}$.
\end{proof}

\begin{proposition}\label{mg-test}
Let $I$ be a monomial
ideal of $S$ and let $A$ be its incidence matrix. A monomial $t^\alpha\in
S$ is a minimal generator of
$\overline{I^n}$ if and only if the following two conditions hold.
\begin{align}
&\max\{\langle y, 1\rangle\mid 
y\geq 0;\, Ay\leq\alpha \}=\min\{\langle\alpha,x\rangle\mid x\geq 0;\,
xA\geq
1\}\geq n.\label{jul10-22-1}\\
&\max\{\langle y, 1\rangle\mid 
y\geq 0;\, Ay\leq\alpha-e_i\}=\min\{\langle\alpha-e_i,x\rangle\mid x\geq 0;\,
xA\geq
1\}<n\label{jul10-22-2}\\
&\mbox{for each }e_i\mbox{ for which }\alpha-e_i\geq 0.\nonumber 
\end{align}
\end{proposition}

\begin{proof} $\Rightarrow$) By Proposition~\ref{membership-test-n},
Eq.~\eqref{jul10-22-1} holds. If Eq.~\eqref{jul10-22-2} does not
hold for some $i$ for which $\alpha-e_i\geq 0$, then by
Proposition~\ref{membership-test-n} one has
$t^{\alpha-e_i}\in\overline{I^n}$ and $t^\alpha=t_it^{\alpha-e_i}$, a
contradiction.

$\Leftarrow$) By Eq.~\eqref{jul10-22-1} and
Proposition~\ref{membership-test-n}, one has
that $t^\alpha\in\overline{I^n}$. We argue by contradiction assuming that
$t^\alpha\notin\mathcal{G}(\overline{I^n})$. Then, there
is $t^\beta\in\mathcal{G}(\overline{I^n})$ such that
$t^\alpha=t^\delta t^\beta$ and $t_i$ divides $t^\delta$ for some $i$.
Hence $t^{\alpha-e_i}\in\overline{I^n}$ and, by
Proposition~\ref{membership-test-n}, 
Eq.~\eqref{jul10-22-2} does not hold, a contradiction. 
\end{proof}

\begin{corollary}\label{n-th-power}
If $I=(t^{v_1},\ldots,t^{v_q})$ is a monomial ideal, then
$\overline{I^n}=\overline{(t^{nv_1},\ldots,t^{nv_q})}$ for $n\geq 1$.
\end{corollary}
\begin{proof} 
Let $A$ be the incidence matrix of $I$. 
We set $J=(t^{nv_1},\ldots,t^{nv_q})$. The inclusion
$\overline{I^n}\supset\overline{J}$ is clear because $I^n\supset J$.
To show the reverse inclusion take $t^\alpha\in\overline{I^n}$. Then,
by Proposition~\ref{membership-test-n}, $A\lambda\leq\alpha$ for some
$\lambda\in\mathbb{Q}_+^q$ with $|\lambda|\geq n$. Hence
$(nA)(\lambda/n)\leq\alpha$ and, by
Proposition~\ref{membership-test-n}, we get that $t^\alpha\in\overline{J}$
because $nA$ is the incidence matrix of $J$ and $|\lambda/n|\geq 1$.
\end{proof}

\begin{corollary}\cite[Corollary~2.5]{poset}\label{crit-rounding-normali} 
Let $I=(x^{v_1},\ldots,x^{v_q})$ be a monomial
ideal and let $A$ be the matrix with column vectors 
$v_1,\ldots,v_q$. Then, $I$ is a normal ideal if and only if 
the system $x\geq 0;xA\geq{1}$ has the integer rounding
property. 
\end{corollary}

\begin{proof} $\Rightarrow$) Assume that $I$ is normal. Let
$\alpha$ be an integral vector for which the 
right-hand side of Eq.~\eqref{irp-eq1} is finite. 
Therefore
\begin{equation*}
{\rm max}\{\langle y,{1}\rangle\mid y\geq 0;
Ay\leq \alpha\}=\langle \lambda,{1}\rangle=|\lambda|
\end{equation*}  
for some $\lambda\in\mathbb{Q}_+^q$ with $A\lambda\leq\alpha$. Note
that $\alpha\in\mathbb{N}^s$. In general one
has
\begin{equation}\label{irp-eq3}
{\rm max}\{\langle y,{1}\rangle\mid y\in\mathbb{N}^q;\,Ay\leq \alpha\} 
\leq\lfloor{\rm max}\{\langle y,{1}\rangle\mid y\geq 0;
Ay\leq \alpha\}\rfloor=\lfloor|\lambda|\rfloor.
\end{equation}  
\quad Then, by Proposition~\ref{normality-criterion-ip}, there is 
$m\in\mathbb{N}^q$ satisfying $Am\leq\alpha$ and
$|\lambda|=|m|+\epsilon$ with $0\leq\epsilon<1$. Thus, the left-hand
side of Eq.~\eqref{irp-eq3} is at least $|m|$, the right-hand
side of Eq.~\eqref{irp-eq3} is $|m|$, and equality holds in
Eq.~\eqref{irp-eq3}. 

$\Leftarrow$) Assume that the system $x\geq0;xA\geq 1$ has the integer
rounding property.  To prove that $I$ is normal take
$t^\alpha\in\overline{I^n}$, $n\geq 1$. Then, by
Proposition~\ref{membership-test-n}, $A\lambda\leq\alpha$ for 
some $\lambda\in\mathbb{Q}_+^q$ with $|\lambda|\geq n$. Therefore
\begin{equation*}
{\rm max}\{\langle y,{1}\rangle \mid y\in\mathbb{N}^q;\, Ay\leq \alpha\} 
=\lfloor{\rm max}\{\langle y,{1}\rangle\mid y\geq 0;
Ay\leq \alpha\}\rfloor\geq\lfloor|\lambda|\rfloor\geq n.
\end{equation*}  
\quad Hence, there is $m\in\mathbb{N}^q$ such that
$Am\leq\alpha$, $|m|\geq n$, and consequently $t^\alpha\in
I^{|m|}\subset I^n$.  
\end{proof}

The \textit{support} of a monomial
$t^a\in S$, $a=(a_1,\ldots,a_s)$, denoted by ${\rm supp}(t^a)$, is
the set of all variables $t_i$ such that $a_i>0$.

\begin{proposition}\label{may21-22} Let $\mathcal{C}$ be a clutter and let
$I_c(\mathcal{C})$ be its ideal of covers. If $I_c(\mathcal{C})$ is
normal, then $I_c(\mathcal{H})$ is normal for any minor $\mathcal{H}$ of $\mathcal{C}$.
\end{proposition}

\begin{proof} It suffices to show that $I_c(\mathcal{C}\setminus v)$ and
$I_c(\mathcal{C}/v)$ are normal for any $v\in V(\mathcal{C})$. 
We set $V(\mathcal{C})=\{t_1,\ldots,t_s\}$,
$\mathcal{H}=\mathcal{C}\setminus v$, $\mathcal{D}=\mathcal{C}/v$,
and $v=t_s$. We may assume that $t_s$ is not an isolated vertex of
$\mathcal{C}$, i.e., there is at least one edge of $\mathcal{C}$ that contains $t_s$.

To prove that $I_c(\mathcal{C}\setminus t_s)$
is normal, we show that $I_c(\mathcal{H})^n=\overline{I_c(\mathcal{H})^n}$ 
for all $n\geq 1$. The inclusion $I_c(\mathcal{H})^n\subset \overline{I_c(\mathcal{H})^n}$ holds in
general. To show the reverse inclusion take
$t^a=t_1^{a_1}\cdots t_{s-1}^{a_{s-1}}\in\overline{I_c(\mathcal{H})^n}$,
$a=(a_1,\ldots,a_{s-1},0)$. Then, there is $k\geq 1$ such that
$(t^a)^k\in I_c(\mathcal{H})^{kn}$ and we can write
\begin{equation*}
(t^a)^k=t^{b_1}\cdots t^{b_{nk}}t^\delta,
\end{equation*}
with $t^{b_i}\in\mathcal{G}(I_c(\mathcal{H}))$ for $i=1,\ldots,kn$. We may
assume that $t^{b_1},\ldots,t^{b_r}$ are in $\mathcal{G}(I_c(\mathcal{C}))$ and
$t^{b_{r+1}},\ldots,t^{b_{kn}}$ are not in $\mathcal{G}(I_c(\mathcal{C}))$. Note
that $t^{b_{r+1}}t_s,\ldots,t^{b_{kn}}t_s$ are in
$\mathcal{G}(I_c(\mathcal{C}))$. Therefore, 
\begin{equation*}
(t^at_s^n)^k=t^{b_1}\cdots t_r^{b_r}(t^{b_{r+1}}t_s)\cdots
(t^{b_{kn}}t_s)t_s^{r} t^\delta,
\end{equation*}
and consequently $(t^at_s^n)^k\in I_c(\mathcal{C})^{kn}$, that is,
$t^at_s^n\in\overline{I_c(\mathcal{C})^n}=I_c(\mathcal{C})^n$. Then, 
$t^at_s^n=t^{c_1}\cdots t^{c_n}t^\gamma$ with
$t^{c_i}\in\mathcal{G}(I_c(\mathcal{C}))$ for $i=1,\ldots,n$. Any minimal vertex
cover of $\mathcal{C}$ contains a minimal vertex cover of $\mathcal{H}$. Hence, we can
write each $t^{c_i}$ as $t^{c_i}=t^{d_i}t^{\epsilon_i}$ with
$t^{d_i}\in\mathcal{G}(I_c(\mathcal{H}))$. From the equality 
\begin{equation*}
t^at_s^n=(t^{d_1}\cdots t^{d_{n}})(t^{\epsilon_1}\cdots
t^{\epsilon_{n}})t^\gamma,
\end{equation*}
we obtain $t^a=t^{d_1}\cdots t^{d_{n}}t^\epsilon$ because
$t_s\notin{\rm supp}(t^{d_i})$ for $i=1,\ldots,n$, and $t^a\in
I_c(\mathcal{H})^n$.

To prove that $I_c(\mathcal{C}/t_s)$
is normal, we show that $I_c(\mathcal{D})^n=\overline{I_c(\mathcal{D})^n}$ 
for all $n\geq 1$. The inclusion $I_c(\mathcal{D})^n\subset
\overline{I_c(\mathcal{D})^n}$ holds in
general. To show the reverse inclusion take
$t^a=t_1^{a_1}\cdots t_{s-1}^{a_{s-1}}\in\overline{I_c(\mathcal{D})^n}$,
$a=(a_1,\ldots,a_{s-1},0)$. Then, there is $k\geq 1$ such that
$(t^a)^k\in I_c(\mathcal{D})^{kn}$ and we can write
\begin{equation*}
(t^a)^k=t^{b_1}\cdots t^{b_{nk}}t^\delta,
\end{equation*}
with $t^{b_i}\in\mathcal{G}(I_c(\mathcal{D}))$ for $i=1,\ldots,kn$.
Let $f_i$ be the support of $t^{b_i}$. Then, either $f_i\in
E(\mathcal{C}^\vee)$ or $f_i=e\setminus{t_s}$ for some $e\in
E(\mathcal{C}^\vee)$ for which $t_s\in e$. Thus, either
$t^{b_i}\in\mathcal{G}(I_c(\mathcal{C}))$ or
$t_st^{b_i}\in\mathcal{G}(I_c(\mathcal{C}))$. Hence, 
$(t^at_s^n)^k\in I_c(\mathcal{C})^{kn}$, that is,
$t^at_s^n\in\overline{I_c(\mathcal{C})^n}=I_c(\mathcal{C})^n$. Then, 
$t^at_s^n=t^{c_1}\cdots t^{c_n}t^\gamma$ with
$t^{c_i}\in\mathcal{G}(I_c(\mathcal{C}))$ for $i=1,\ldots,n$. Making 
$t_s=1$, it follows readily that $t^a\in I_c(\mathcal{D})^n$ because
each $t^{c_i}$ is divisible by some monomial $t^{u_i}$ in
$\mathcal{G}(I_c(\mathcal{D}))$.
\end{proof}

\section{Ideals generated by monomials of degree
$2$}\label{normality-degree2}

In this section we use Hilbert bases, Ehrhart
rings of lattice polytopes, and integer rounding properties, to study the
normality of ideals generated by monomials of degree $2$. 
Throughout this section $I=(t^{v_1},\ldots,t^{v_q})$ is an ideal
generated by monomials of degree $2$ and $A$ is the incidence matrix
of $I$. Let $\mathcal{B}$, $\mathcal{Q}$, $R=K[\mathbb{N}\mathcal{B}]$, and ${\rm
Er}(\mathcal{Q})$ be as in Eqs.~\eqref{jul20-22}--\eqref{jul20-22-4}.

\begin{theorem}\label{normal-b} 
Let $I=(t^{v_1},\ldots,t^{v_q})$ 
be an ideal of $S$ generated by monomials
of degree $2$. Then, $I$ is a normal ideal if and only if
$\mathcal{B}$ is a Hilbert basis.
\end{theorem}

\begin{proof} $\Rightarrow$) Assume that $I$ is normal. The
inclusion
$\mathbb{R}_+\mathcal{B}\cap\mathbb{Z}^{s+1}\supset\mathbb{N}\mathcal{B}$
holds in general. To show the reverse inclusion take
$(\alpha,b)\in\mathbb{R}_+\mathcal{B}\cap\mathbb{Z}^{s+1}$,
$\alpha\in\mathbb{N}^s$, $b\in\mathbb{N}$. By Farkas's lemma
(Theorem~\ref{farkas}), we obtain that $(\alpha,b)$ is in
 $\mathbb{Q}_+\mathcal{B}$, that is, we can write
\begin{equation}\label{jun12-22-1}
(\alpha,b)=\tau_1
e_{s+1}+\sum_{i=1}^s\mu_i(e_i+e_{s+1})+\sum_{i=1}^q\lambda_i(v_i,1),
\end{equation}
where $\tau_1,\mu_i$, and $\lambda_i$ are in $\mathbb{Q}_+$. Then,
setting $\lambda=(\lambda_1,\ldots,\lambda_q)$ and
$\mu=(\mu_1,\ldots,\mu_s)$, one has
$A\lambda\leq\alpha$ and $b\geq |\mu|+|\lambda|$. Hence, by
Proposition~\ref{normality-criterion-ip}, there is 
$m=(m_1,\ldots,m_q)\in\mathbb{N}^q$ satisfying $Am\leq\alpha$ and
$|\lambda|=|m|+\epsilon$ with $0\leq\epsilon<1$. Thus  
\begin{equation}\label{jun12-22-2}
\alpha=\sum_{i=1}^s c_ie_i+\sum_{i=1}^qm_iv_i,
\end{equation}
where $c_1,\ldots,c_s$ are in $\mathbb{N}$. Setting
$c=(c_1,\ldots,c_s)$, from Eqs.~\eqref{jun12-22-1} and
\eqref{jun12-22-2}, we get 
\begin{equation}\label{jun12-22-3}
|\mu|+2|\lambda|=|c|+2|m|.
\end{equation}
\quad Therefore, using that $|\lambda|=|m|+\epsilon$, $b\geq
|\mu|+|\lambda|$, and Eq.~\eqref{jun12-22-3}, one has
$$
b+(|m|+\epsilon)=b+|\lambda|\geq
(|\mu|+|\lambda|)+|\lambda|=|c|+2|m|,
$$
and consequently $b\geq |c|+|m|-\epsilon$. We claim that
$b\geq|c|+|m|$. We argue by contradiction assuming that $b<|c|+|m|$. 
Then, $|c|+|m|-\epsilon\leq b\leq |c|+|m|-1$, and we obtain
that $\epsilon\geq 1$, a contradiction. Then, by Eq.~\eqref{jun12-22-2}, we can write
$$
(\alpha,b)=(b-|c|-|m|)e_{s+1}+\sum_{i=1}^sc_i(e_i+e_{s+1})+\sum_{i=1}^qm_i(v_i,1),
$$
and $(\alpha,b)\in\mathbb{N}\mathcal{B}$. This completes the proof
that $\mathcal{B}$ is a Hilbert basis.

$\Leftarrow$) Assume that $\mathcal{B}$ is a Hilbert basis and
set $\mathcal{A}'=\{e_i\}_{i=1}^s\cup\{(v_i,1)\}_{i=1}^q$. To show the
normality of $I$ we need only show that
$\mathbb{R}_+\mathcal{A}'\cap\mathbb{Z}^{s+1}=\mathbb{N}\mathcal{A}'$
(see Lemma~\ref{normal-hilbert}). The
inclusion
$\mathbb{R}_+\mathcal{A}'\cap\mathbb{Z}^{s+1}\supset\mathbb{N}\mathcal{A}'$
holds in general. To show the reverse inclusion take
$(\alpha,b)\in\mathbb{R}_+\mathcal{A}'\cap\mathbb{Z}^{s+1}$,
$\alpha\in\mathbb{N}^s$, $b\in\mathbb{N}$. Then, we can write 
\begin{equation}\label{jun14-22-1}
(\alpha,b)=\sum_{i=1}^s\mu_ie_i+\sum_{i=1}^q\lambda_i(v_i,1),
\end{equation}
where $\mu_i$ and $\lambda_i$ are in $\mathbb{R}_+$. Hence,
setting $\mu=(\mu_1,\ldots,\mu_s)$ and
$\lambda=(\lambda_1,\ldots,\lambda_q)$, one has
$|\alpha|=|\mu|+2|\lambda|=|\mu|+2b$. Noticing that $|\mu|=|\alpha|-2b$ is in
$\mathbb{N}$, by Eq.~\eqref{jun14-22-1}, we obtain that 
$(\alpha,b)+|\mu|e_{s+1}$ is in
$\mathbb{R}_+\mathcal{B}\cap\mathbb{Z}^{s+1}=\mathbb{N}\mathcal{B}$. Therefore, we can write
\begin{equation}\label{jun14-22-2}
(\alpha,b)+|\mu|e_{s+1}=ne_{s+1}+\sum_{i=1}^sc_i(e_i+e_{s+1})+\sum_{i=1}^qm_i(v_i,1),
\end{equation}
where $n$, $c_i$, and $m_i$ are in $\mathbb{N}$. Setting
$c=(c_1,\ldots,c_s)$ and $m=(m_1,\ldots,m_q)$, 
from Eqs.~\eqref{jun14-22-1} and
\eqref{jun14-22-2}, we get the equalities 
\begin{align}
|\alpha|&=|\mu|+2b=|c|+2|m|,\label{jun14-22-3}\\ 
b&=|\lambda|=n+|c|+|m|-|\mu|.\label{jun14-22-4}
\end{align}
\quad Then, from Eq.~\eqref{jun14-22-3} and \eqref{jun14-22-4}, one has 
\begin{equation*}
|\mu|+2b=|c|+2|m|=(b-n-|m|+|\mu|)+2|m|=b-n+|\mu|+|m|,
\end{equation*}
and consequently
$|\mu|+2b=b-n+|\mu|+|m|$. Thus, $b+n=|m|$. Therefore, adding $n$ to both sides of
Eq.~\eqref{jun14-22-4}, we obtain
\begin{equation}\label{jun14-22-5}
|m|=b+n=2n+|c|+|m|-|\mu|\ \therefore\ 2n+|c|-|\mu|=0.
\end{equation}
\quad Consider the multiset 
$$
\mathcal{F}=\{\underbrace{(v_1,1),\ldots,(v_1,1)}_{m_1\tiny\mbox{
times }},\ldots,\underbrace{(v_q,1),\ldots,(v_q,1)}_{m_q\tiny\mbox{
times }}\}.
$$
\quad This multiset has $|m|=n+b$ elements. Pick a multiset 
$\mathcal{F}_1=\{(v_{\ell_1},1),\ldots,(v_{\ell_b},1)\}$ of $b$ vectors
in $\mathcal{F}$ and let
$\mathcal{F}_2=\mathcal{F}\setminus\mathcal{F}_1
=\{(v_{j_1},1),\ldots,(v_{j_n},1)\}$ be the complement of
$\mathcal{F}_1$. Then, by Eqs.~\eqref{jun14-22-2} and
\eqref{jun14-22-5}, one has
\begin{align*}
(\alpha,b)&=(n-|\mu|)e_{s+1}+\sum_{i=1}^sc_i(e_i+e_{s+1})+\sum_{i=1}^qm_i(v_i,1)\\
          &=(n-|\mu|+|c|)e_{s+1}+\sum_{i=1}^sc_ie_i+\sum_{i=1}^n(v_{j_i},1)+
\sum_{i=1}^b(v_{\ell_i},1)\\
         &=(2n-|\mu|+|c|)e_{s+1}+
\sum_{i=1}^sc_ie_i+\sum_{i=1}^nv_{j_i}+\sum_{i=1}^b(v_{\ell_i},1)\\
         &=\sum_{i=1}^sc_ie_i+\sum_{i=1}^nv_{j_i}+\sum_{i=1}^b(v_{\ell_i},1).
\end{align*}
\quad Thus, $(\alpha,b)$ is in $\mathbb{N}\mathcal{A}'$ and the proof
is complete.
\end{proof}

\begin{corollary}\cite[Theorem~3.3]{roundp}\label{jul28-22} 
If $I=(t^{v_1},\ldots,t^{v_q})$ is the edge ideal
of a connected graph, then $I$ is normal if and only if the subring 
$R=K[z,t_1z,\ldots,t_sz,t^{v_1}z,\ldots,t^{v_q}z]$ is normal.
\end{corollary}

\begin{proof} The subgroup
$\mathbb{Z}\mathcal{B}$ 
spanned by $\mathcal{B}$ is $\mathbb{Z}^{n+1}$. 
Then, by Theorem~\ref{int-clos-gen}, one
has 
$$
R=K[\{t^az^b\vert\, (a,b)\in\mathbb{N}{\mathcal B}\}]\subset\overline{R}=K[\{t^az^b\vert\, (a,b)\in
\mathbb{Z}^{s+1}\cap \mathbb{R}_+{\mathcal B}\}].
$$
\quad Hence, $R$ is normal if and only if 
$\mathbb{N}{\mathcal B}$ is equal to $\mathbb{Z}^{s+1}\cap\mathbb{R}_+{\mathcal B}$.
Thus, the result follows from Theorem~\ref{normal-b} because $I$ is
generated by squarefree monomials of degree $2$.
\end{proof}

\begin{lemma}\label{lemma-a<=b}
Let $I=(t^{v_1},\ldots,t^{v_q})$ be an ideal 
of $S$ generated by monomials of degree $2$ and let
$\mathcal{A}$ be the set of vectors
$\{e_i\}_{i=1}^s\cup\{v_i\}_{i=1}^q$. The following hold.
\begin{enumerate}
\item[(a)] If 
$\alpha$ is in $\mathbb{N}^s\setminus\{0\}$, $\beta_1,\ldots,\beta_k$ are
in $\mathcal{A}$ and $\sum_{i=1}^k\beta_i\geq\alpha$, 
then there are $\gamma_1,\ldots,\gamma_\ell$ in $\mathcal{A}$
such that $\alpha=\sum_{i=1}^\ell\gamma_i$ and $k\geq\ell$.
\item[(b)] If $a\in\mathbb{Q}_+^s$, 
$b\in{\rm conv}(\{0\}\cup\mathcal{A})\cap\mathbb{Q}^s$ and $a\leq b$, 
then $a\in{\rm conv}(\{0\}\cup\mathcal{A})$.
\end{enumerate}
\end{lemma}

\begin{proof} (a) Consider the following procedure. Assume that 
$\sum_{i=1}^k\beta_i\neq\alpha$. Then, the $j$-th entry of
$\sum_{i=1}^k\beta_i$ is greater than the $j$-th entry of 
$\alpha$ for some $j$, and consequently
$\sum_{i=1}^k\beta_i\geq\alpha+e_j$. Hence, $\beta_p\geq e_j$ for 
some $p$, and either $\beta_p=e_j$ or $\beta_p-e_j=e_r$ for some
$1\leq r\leq s$.
Thus
$$
\left(\sum_{i=1}^k\beta_i\right)-e_j=\bigg(\sum_{i\neq
p}\beta_i\bigg)+(\beta_p-e_j)\geq\alpha,
$$
where $\beta_p-e_j=0$ or $\beta_p-e_j\in\mathcal{A}$. If $\big(\sum_{i\neq
p}\beta_i\big)+(\beta_p-e_j)\neq\alpha$, we repeat the procedure.
Since 
$$
\bigg|\bigg(\sum_{i\neq
p}\beta_i\bigg)+(\beta_p-e_j)\bigg|<\left|\sum_{i=1}^k\beta_i\right|,
$$
applying this procedure recursively, we get that 
$\alpha=\sum_{i=1}^\ell\gamma_i$ for some $\gamma_1,\ldots,\gamma_\ell$ in $\mathcal{A}$
and $k\geq\ell$.

(b) One can write $b=\sum_{i=1}^s\mu_ie_i+\sum_{i=1}^q\lambda_iv_i$, 
$\sum_{i=1}^s\mu_i+\sum_{i=1}^q\lambda_i\leq 1$, $\mu_i$, $\lambda_j$
in $\mathbb{Q}_+$ for all $i,j$. If $a=0$, there is nothing to prove.
Assume that $a\neq 0$.  Choose $r\in\mathbb{N}_+$ such that $r\mu_i$,
$r\lambda_j$ are in $\mathbb{N}$ for all $i,j$ and
$ra\in\mathbb{N}^s\setminus\{0\}$. Then
$$
ra\leq
rb=\sum_{i=1}^s(r\mu_i)e_i+\sum_{i=1}^q(r\lambda_i)v_i,
$$
and $k:=\sum_{i=1}^s(r\mu_i)+\sum_{i=1}^q(r\lambda_i)\leq r$. Then, by
part (a), there are $\gamma_1,\ldots,\gamma_\ell$ in $\mathcal{A}$
such that $ra=\sum_{i=1}^\ell\gamma_i$, $\ell\leq k\leq r$, and
$a=\sum_{i=1}^\ell(\gamma_i/r)$. Hence, as $\ell/r\leq 1$, we get
$a\in{\rm conv}(\{0\}\cup\mathcal{A})$.
\end{proof}

\begin{theorem}\label{ehrhart-q} 
Let $I$ be an ideal 
of $S$ generated by monomials of degree $2$ and
let $A$ be the incidence matrix of $I$. Then, 
$\overline{K[\mathbb{N}\mathcal{B}]}={\rm Er}(\mathcal{Q})$, and the equality 
$$ 
K[\mathbb{N}\mathcal{B}]={\rm Er}(\mathcal{Q})
$$
holds if and only if the 
system $x\geq 0;\, xA\leq 1$ has the integer
rounding property.
\end{theorem}

\begin{proof} We may assume that
$\{t_1,\ldots,t_s\}=\bigcup_{i=1}^q{\rm supp}(t^{v_i})$, i.e., each
variable $t_i$ occurs in at least one minimal generator of $I$. The equality
$\overline{K[\mathbb{N}\mathcal{B}]}={\rm Er}(\mathcal{Q})$ follows readily
from \cite[Theorem~3.9]{ehrhart}.

$\Rightarrow$) Assume that $K[\mathbb{N}\mathcal{B}]={\rm
Er}(\mathcal{Q})$. Let $\alpha$ be an integral vector for
which the left hand side 
of Eq.~\eqref{irp-eq2} is finite. By replacing $\alpha$ by its
positive part $\alpha_+$, we may assume that
$\alpha\in\mathbb{N}^s\setminus\{0\}$.
In general one has
\begin{equation}\label{jun30-22-1}
\lceil|\lambda|\rceil=\lceil{\rm min}\{\langle y,{1}\rangle \vert\, 
y\geq 0;\, Ay\geq \alpha \}\rceil
\leq {\rm min}\{\langle y,{1}\rangle \mid  
y\in\mathbb{N}^q;\, Ay\geq \alpha\},
\end{equation}
where $\lambda=(\lambda_1,\ldots,\lambda_q)\in\mathbb{Q}_+^q$ and
$A\lambda\geq\alpha$. Then
$$
\frac{\alpha}{\lceil|\lambda|\rceil}\leq\frac{\alpha}{|\lambda|}
\leq\sum_{i=1}^q\frac{\lambda_i}{|\lambda|}v_i
$$
and, by Lemma~\ref{lemma-a<=b}, we obtain
that $\alpha/\lceil|\lambda|\rceil\in\mathcal{Q}$, that is,
$t^{\alpha}z^{\lceil|\lambda|\rceil}\in{\rm
Er}(\mathcal{Q})=K[\mathbb{N}\mathcal{B}]$. Therefore, there are
nonnegative integers $\tau_1,n_i,m_j\in\mathbb{N}$ such that
\begin{align}
t^{\alpha}z^{\lceil|\lambda|\rceil}&=z^{\tau_1}(t_1z)^{n_1}
\cdots(t_sz)^{n_s}(t^{v_1}z)^{m_1}\cdots(t^{v_q}z)^{m_q},\\
\lceil|\lambda|\rceil&=\tau_1+n_1+\cdots+n_s+m_1+\cdots+m_q,\label{jun30-22-2}\\
\alpha&=n_1e_1+\cdots+n_se_s+m_1v_1+\cdots+m_qv_q.\label{jun30-22-5}
\end{align}
\quad For each $e_i$ there is $v_{j_i}\in\{v_1,\ldots,v_q\}$
satisfying $e_i\leq v_{j_i}$. Then, by Eq.~\eqref{jun30-22-5}, $\alpha\leq Aw$ for some
$w\in\mathbb{N}^q$ with $|w|=\sum_{i=1}^sn_i+\sum_{i=1}^qm_i$. From
Eqs.~\eqref{jun30-22-1} and \eqref{jun30-22-2}, we get
$$
|w|\leq\lceil|\lambda|\rceil=\lceil{\rm min}\{\langle y,{1}\rangle
\mid y\geq 0;\, Ay\geq \alpha \}\rceil
\leq {\rm min}\{\langle y,{1}\rangle \mid y\in\mathbb{N}^q;\, Ay\geq \alpha\}\leq |w|,
$$ 
and we have equality everywhere. Thus, the 
system $x\geq 0;\, xA\leq 1$ has the integer
rounding property and the proof of this implication is complete.

$\Leftarrow$) Assume that the linear system $x\geq 0;\, xA\leq 1$ has the integer
rounding property. The inclusion $K[\mathbb{N}\mathcal{B}]\subset{\rm
Er}(\mathcal{Q})$ is clear because
$\overline{K[\mathbb{N}\mathcal{B}]}={\rm Er}(\mathcal{Q})$. To show the reverse inclusion take
$t^\alpha z^b\in{\rm Er}(\mathcal{Q})$, that is,
$\alpha\in b\mathcal{Q}$, $\alpha\in\mathbb{N}^s$,
$b\in\mathbb{N}_+$.
Then, 
$$\alpha/b=\sum_{i=1}^s\mu_ie_i+\sum_{i=1}^q\lambda_iv_i,
$$
where $\mu_i$, $\lambda_j$ are in $\mathbb{R}_+$, and
$\sum_{i=1}^s\mu_i+\sum_{i=1}^q\lambda_i\leq 1$. For any vector $x$ 
that satisfies $x\geq 0;\, xA\leq 1$, one has $\langle x,e_i
\rangle\leq 1$ for $i=1,\ldots,s$ because any $t_i$ occurs in at
least one of the
minimal generators of $I$. Hence, for any such $x$, we obtain
$$
\langle\alpha/b,x \rangle=\sum_{i=1}^s\mu_i\langle e_i,x \rangle+
\sum_{i=1}^q\lambda_i\langle v_i,x \rangle\leq\sum_{i=1}^s\mu_i+
\sum_{i=1}^q\lambda_i\leq 1. 
$$
\quad Thus, $\langle\alpha,x \rangle\leq b$ and, by linear
programming duality \cite[Corollary~7.1g]{Schr}, one has
\begin{equation}\label{jun30-22-3}
b\geq\max\{\langle\alpha, x\rangle\mid 
x\geq 0;\, xA\leq 1 \}=\min\{\langle y,{1}\rangle\mid y\geq 0;\,Ay\geq
\alpha\},
\end{equation}
and since system $x\geq 0;\, xA\leq 1$ has the integer rounding
property, we get 
\begin{equation}\label{jun30-22-4}
b\geq \lceil{\rm min}\{\langle y,{1}\rangle \mid y\geq 0;\, Ay\geq \alpha \}\rceil
={\rm min}\{\langle y,{1}\rangle \mid y\in\mathbb{N}^q;\, Ay\geq \alpha\}.
\end{equation}
\quad Hence, we can choose $m\in\mathbb{N}^q$ such that $b\geq|m|$ and
$Am\geq\alpha$. Setting $k=|m|$, by Lemma~\ref{lemma-a<=b}, there are
$\gamma_1,\ldots,\gamma_\ell$ in $\{e_i\}_{i=1}^s\cup\{v_i\}_{i=1}^q$,
$|m|\geq\ell$, such that $\alpha=\sum_{i=1}^\ell\gamma_i$. Thus,
$$
t^\alpha z^b=(t^{\gamma_1}z)\cdots(t^{\gamma_\ell}z)z^{b-\ell},
$$
and consequently $t^\alpha z^b\in K[\mathbb{N}\mathcal{B}]$. 
\end{proof}

\begin{corollary}\label{ehrhart-corollary}
Let $I$ be an ideal of $S$ generated by monomials of degree $2$. Then,
the following conditions are
equivalent. 
$$
\mathrm{(a)}\ K[\mathbb{N}\mathcal{B}]={\rm Er}(\mathcal{Q});
\quad \mathrm{(b)}\ K[\mathbb{N}\mathcal{B}]\text{ is normal };\quad 
\mathrm{(c)}\ \mathcal{B}\text{ is a Hilbert basis}.
$$
\end{corollary}

\begin{proof} (a)$\Rightarrow$(b) This follows from the fact that 
the Ehrhart ring of a lattice polytope is a normal
domain \cite[Theorem~9.3.6]{monalg-rev}.

(b)$\Rightarrow$(c) Noticing that $\mathbb{Z}\mathcal{B}$ is
$\mathbb{Z}^{s+1}$, by the description of the integral closure of
$K[\mathbb{N}\mathcal{B}]$ given in Theorem~\ref{int-clos-gen}, 
one has 
$$K[\mathbb{N}\mathcal{B}]=\overline{K[\mathbb{N}\mathcal{B}]}=K[
\mathbb{Z}\mathcal{B}\cap \mathbb{R}_+{\mathcal
B}]=K[\mathbb{Z}^{s+1}\cap \mathbb{R}_+{\mathcal B}].
$$
\quad Thus, $\mathbb{N}\mathcal{B}=\mathbb{Z}^{s+1}\cap
\mathbb{R}_+{\mathcal B}$ and $\mathcal{B}$ is a Hilbert basis.

(c)$\Rightarrow$(a) By Theorem~\ref{ehrhart-q}, we get
$\overline{K[\mathbb{N}\mathcal{B}]}={\rm Er}(\mathcal{Q})$. Hence,
using that $\mathcal{B}$ is a Hilbert basis and the description of the integral closure of
$K[\mathbb{N}\mathcal{B}]$ given in Theorem~\ref{int-clos-gen},
one has
$$\overline{K[\mathbb{N}\mathcal{B}]}=K[
\mathbb{Z}\mathcal{B}\cap \mathbb{R}_+{\mathcal
B}]=K[\mathbb{Z}^{s+1}\cap \mathbb{R}_+{\mathcal
B}]=K[\mathbb{N}\mathcal{B}].
$$
\quad Thus, ${K[\mathbb{N}\mathcal{B}]}={\rm Er}(\mathcal{Q})$ and the
proof is complete.
\end{proof}

\begin{corollary}\label{jul4-22} Let $I$ be the edge ideal of a graph
and let $I^*$ be the dual of $I$.  
Then, $I^*$ is normal if and only if
$\mathcal{B}$ is a Hilbert basis.
\end{corollary}

\begin{proof} Let $A$ be the incidence matrix of $I$. By
Corollary~\ref{crit-rounding-normali}, $I^*$ is normal if and only if
the system $xA^*\geq 1;x\geq 0$ has the integer rounding property.
Then, by Theorem~\ref{duality-irp}, $I^*$ is normal if and only if the
system $xA\leq 1;x\geq 0$ has the integer rounding property. Hence, by
Theorem~\ref{ehrhart-q}, $I^*$ is normal if and only if 
${K[\mathbb{N}\mathcal{B}]}={\rm Er}(\mathcal{Q})$. Thus, by
Corollary~\ref{ehrhart-corollary}, $I^*$ is normal if and only if
$\mathcal{B}$ is a Hilbert basis.
\end{proof}

\begin{proposition}\label{duality-i-i*} Let $I$ be the edge ideal of a graph and let
$I^*$ be the dual of $I$. Then, $I^*$ is normal if and only if $I$ is
normal.
\end{proposition}

\begin{proof} This follows from Theorem~\ref{normal-b} and
Corollary~\ref{jul4-22}.
\end{proof}

\section{Normality of ideals of covers of
graphs}\label{normality-icdual}

In this section, we study the normality of the ideal
of covers $I_c(G)$ of a graph $G$ and give a combinatorial criterion
in terms of Hochster configurations for the normality of $I_c(G)$
when the independence number of $G$ is at most 
two.

\begin{lemma}\label{may13-22}
Let $v$ be a non isolated vertex of a graph $G$. If the neighbor set $N_G(v)$ of
$v$ is a minimal 
vertex cover of $G$, then a set $C\subset V(G)$ is a
minimal vertex cover of $G$ if and only if $C=N_G(v)$ or $C=\{v\}\cup
D$ with $D$ a minimal vertex cover of $G\setminus v$ such that 
$N_G(v)\not\subset D$.
\end{lemma}
\begin{proof} $\Rightarrow$) Assume that $C$ is a minimal vertex
cover of $G$. If $v\notin C$, then $N_G(v)\subset C$ and, by the
minimality of $C$, one has the equality $C=N_G(v)$. Now assume that  $v\in C$ and
set $D=C\setminus\{v\}$. If $N_G(v)\subset D$, we get that $D$ is a
vertex cover of $G$ with $D\subsetneq C$, a contradiction. Thus,
$N_G(v)\not\subset D$ and the proof reduces to showing that $D$ is a
minimal vertex cover of $H=G\setminus v$. Take $f\in E(H)$, then
$v\notin f$ and $f\cap
C\neq\emptyset$. Thus, $f\cap D\neq\emptyset$, and $D$ is a vertex
cover of $H$. To show that $D$ is minimal take $t_k\in D$. Then
$t_k\neq v$ and, by
the minimality of $C$, there is $e\in E(G)$ such that
$e\cap(C\setminus\{t_k\})=\emptyset$. Then, $v\notin e$, $e\in E(H)$,
and $e\cap(D\setminus\{t_k\})=\emptyset$.

$\Leftarrow$) Assume that $C=\{v\}\cup D$ with $D$ a minimal vertex
cover of $H=G\setminus v$ such that 
$N_G(v)\not\subset D$. Take $e\in E(G)$. If $v\in e$, then, $e\cap
C\neq\emptyset$ and if $v\notin e$, then $e\cap D\neq\emptyset$. Thus,
$C$ is a vertex cover of $G$. Next we show that $C$ is minimal. As $N_G(v)\not\subset D$, it follows that 
$D=C\setminus\{v\}$ is not a
vertex cover of $G$. Indeed, pick $t_k\in N_G(v)\setminus D$, then
$e=\{v,t_k\}\in E(G)$ and $e\cap(C\setminus\{v\})=\emptyset$. 
Now take $t_i\in C$, $t_i\neq v$. Then, there is
$f\in E(H)$ such that $f\cap(D\setminus\{t_i\})=\emptyset$ because $D$
is a minimal vertex cover of $H$. Then,
$f\cap(C\setminus\{t_i\})=\emptyset$.
\end{proof}

\begin{theorem}\label{may14-22} Let $v$ be a vertex of a graph $G$. If
the neighbor set $N_G(v)$ of $v$ is a minimal vertex cover of $G$,
then $I_c(G\setminus v)$ is normal if and only if $I_c(G)$ is normal. 
\end{theorem}

\begin{proof} Setting $V(G)=\{t_1,\ldots,t_s\}$ and $H=G\setminus v$,
we may assume that $v$ is not an isolated vertex, $v=t_s$, and
$N_G(t_s)=\{t_1,\ldots,t_r\}$.

$\Rightarrow$) Assume that $I_c(H)$ is a normal ideal. To prove that $I_c(G)$
is normal, we will show that $I_c(G)^n=\overline{I_c(G)^n}$ 
for all $n\geq 1$. We argue by induction on $n$. The case $n=1$ is clear
because $I_c(G)$ is squarefree \cite[p.~153]{monalg-rev}. Assume that
$n>1$. The inclusion $I_c(G)^n\subset \overline{I_c(G)^n}$ holds in
general. To show the reverse inclusion take
$t^a\in\overline{I_c(G)^n}$. Then, by Lemma~\ref{icd}, there is
$k\in\mathbb{N}_+$ such that
$(t^a)^k\in I_c(G)^{kn}$ and we can write
\begin{equation*}
(t^a)^k=t^{b_1}\cdots t^{b_{nk}}t^\delta,
\end{equation*}
where $t^{b_1},\ldots, t^{b_{nk}}$ are minimal generators of $I_c(G)$.
Then, by Lemma~\ref{may13-22}, we get 
\begin{equation}\label{may16-22-2}
(t^a)^k=(t_1\cdots t_r)^m(t_st^{d_{m+1}})\cdots (t_st^{d_{nk}})t^\delta,
\end{equation}
where $t^{d_{m+1}},\ldots,t^{d_{nk}}$ are minimal generators of
$I_c(H)$. As $\{t_1,\ldots,t_r\}$ contains a minimal vertex cover of
$H$, one has $(t^a)^k\in I_c(H)^{nk}$, and consequently 
$t^a\in\overline{I_c(H)^n}$. By the normality of $I_c(H)$,
we obtain that $t^a\in I_c(H)^n$. Then
\begin{equation}\label{may16-22-3}
t^a=t^{c_1}\cdots t^{c_{n}}t^\gamma,
\end{equation}
where $t^{c_1},\ldots, t^{c_{n}}$ are minimal generators of $I_c(H)$
and $t^\gamma \in S=K[t_1,\ldots,t_s]$.

Case (I) $n\leq a_s$, $a=(a_1,\ldots,a_s)$. By
Eq.~\eqref{may16-22-3}, $t_s^{a_s}$ divides $t^\gamma$, and we get
\begin{equation*}
t^a=(t_st^{c_1})\cdots(t_st^{c_{n}})(t^\gamma/t_s^{a_s})t^{a_s-n}.
\end{equation*}
\quad Then, $t^{a}\in I_c(G)^n$ because $t_st^{c_i}\in I_c(G)$ for
$i=1,\ldots,n$.

Case (II) $n>a_s$. By Eq.~\eqref{may16-22-2}, one has $ka_s\geq kn-m$
and $ka_i\geq m$ for $i=1,\ldots,r$. Therefore, $m\geq k(n-a_s)$ and
$a_i\geq n-a_s$ for $i=1,\ldots,r$. Thus, we can write
\begin{equation}\label{may16-22-4}
t^a=(t_1\cdots t_r)^{n-a_s}t^\epsilon.
\end{equation}
\quad Using Eqs.~\eqref{may16-22-2} and \eqref{may16-22-4}, we obtain
$$
(t^a)^k=(t_1\cdots t_r)^m(t_st^{d_{m+1}})\cdots
(t_st^{d_{nk}})t^\delta=(t_1\cdots t_r)^{k(n-a_s)}(t^\epsilon)^k,
$$
and consequently 
$$(t_1\cdots t_r)^{m-k(n-a_s)}(t_st^{d_{m+1}})\cdots
(t_st^{d_{nk}})t^\delta=(t^\epsilon)^k.
$$ 
\quad As $m-k(n-a_s)+kn-m=ka_s$, we obtain that $(t^\epsilon)^k\in
I_c(G)^{ka_s}$, that is, $t^\epsilon\in\overline{I_c(G)^{a_s}}$. By
induction ${I_c(G)^{a_s}}$ is equal to $\overline{I_c(G)^{a_s}}$. Thus
$t^\epsilon\in{I_c(G)^{a_s}}$ and, by
Eq.~\eqref{may16-22-4}, we get $t^a\in I_c(G)^n$. Hence,
${I_c(G)^{n}}$ is equal to 
$\overline{I_c(G)^{n}}$ and the proof is
complete.

$\Leftarrow$) This implication follows at once from
Proposition~\ref{may21-22}.
\end{proof}

Let $G$ be a graph. The \textit{cone} over $G$ with apex $v$, denoted
$C(G)$, is obtained
 by adding a new vertex $v$ to $G$ and joining every vertex of $G$ to
 $v$. 

\begin{corollary}\cite[Theorem~1.6]{Al-Ayyoub-Nasernejad-cover-ideals}\label{aug19-22}
Let $G$ be a graph and let $C(G)$ the cone over $G$ with apex
$v$. Then, $I_c(G)$ is normal if and only if $I_c(C(G))$ is normal. 
\end{corollary}

\begin{proof} This follows readily from Theorem~\ref{may14-22}, by noticing
that $V(G)$ is a minimal vertex cover of $C(G)$, $N_{C(G)}(v)=V(G)$,
and $C(G)\setminus v=G$.
\end{proof}

\begin{corollary}\label{may20-22} Let $G$ be a graph and let
$G_1,\ldots,G_r$ be its connected components. Then
\begin{enumerate}
\item[(a)] $I_c(G)=I_c(G_1)\cdots I_c(G_r)$,
\item[(b)] $\overline{I_c(G)^n}=\overline{I_c(G_1)^n}\cdots
\overline{I_c(G_r)^n}$
for all $n\geq 1$, and 
\item[(c)] $I_c(G)$ is normal if and only if $I_c(G_i)$ is normal 
for $i=1,\ldots,r$.
\end{enumerate}
\end{corollary}

\begin{proof} To show part (a), let $C$ be a set of vertices of $G$. Note that $C$
is a minimal vertex cover of $G$ if and only if $C=C_1\cup\cdots\cup
C_r$ with $C_i$ a minimal vertex cover of $G_i$ for $i=1,\ldots,r$. 
Hence, $I_c(G)$ is equal to $I_c(G_1)\cdots I_c(G_r)$. 
Parts (b) and (c) follow from part (a) and Proposition~\ref{may19-22}.
\end{proof}

A {\it clique\/} of a graph $G$
is a set of vertices inducing a complete subgraph. The {\it clique
clutter\/} of $G$, denoted by ${\rm cl}(G)$, is the clutter on $V(G)$ whose edges are the 
maximal cliques of $G$ (maximal with respect to
inclusion). We also call a complete subgraph of $G$ 
a clique and denote a complete
subgraph of $G$ with $r$ vertices by $\mathcal{K}_r$. 
Then
$$
I({\rm cl}(G))^*=(\{(t_1\cdots t_s)/t_e\mid e\in E({\rm cl}(G))\}).
$$
\quad If $G$ is a discrete graph, by convention $I_c(G)=S$, 
$I(G)=(0)$, and $I(G)^*=(0)$. 

The \textit{complement} of a graph $G$, denoted $\overline{G}$, has the same
vertex set as $G$, 
and $\{t_i,t_j\}$ is an edge of
$\overline{G}$ if and only if $\{t_i,t_j\}$ is not an
edge of $G$.

\begin{lemma}\label{jul11-22} Let $\overline{G}$ be the complement of a graph $G$,
let 
${\rm cl}(\overline{G})$ be the clique
clutter of $\overline{G}$, and let ${\rm Isol}(\overline{G})$ be 
the set of isolated vertices of $\overline{G}$. The following hold.
\begin{enumerate}
\item[(a)] $I_c(G)=I({\rm cl}(\overline{G}))^*$.
\item[(b)] If $\overline{G}$ has no triangles, then 
$I_c(G)=(\{(t_1\cdots t_s)/t_i\mid t_i\in{\rm
Isol}(\overline{G})\})+I(\overline{G})^*$.
\item[(c)] If $\overline{G}$ is a discrete graph, then 
$I_c(G)=(\{(t_1\cdots t_s)/t_i\mid t_i\in V(G)\})$.
\item[(d)] If $\overline{G}$ has no triangles and no isolated
vertices, then $I_c(G)=I(\overline{G})^*$.
\end{enumerate}
\end{lemma}

\begin{proof} (a) To show the inclusion ``$\subset$'' take $t^a$ a
minimal generator of $I_c(G)$, i.e., the support $U$
of $t^a$ is a minimal vertex cover of $G$, and consequently
$V(G)\setminus U$ is a maximal clique of $\overline{G}$. Thus, $U$ is the
complement of a maximal clique of $\overline{G}$, and consequently $t^a\in I({\rm
cl}(\overline{G}))^*$. The inclusion ``$\supset$'' is also easy to prove.

(b) As the graph $\overline{G}$ has no triangles, the edges of the
clique clutter ${\rm
cl}(\overline{G})$ are either isolated vertices of $\overline{G}$
(i.e., maximal cliques that
correspond to $\mathcal{K}_1$) or edges of $\overline{G}$ (i.e.,
maximal cliques that correspond to $\mathcal{K}_2$). Hence
$$
I({\rm cl}(\overline{G}))=({\rm
Isol}(\overline{G}))+I(\overline{G}),
$$
and using part (a) we obtain the equalities
$$
I_c(G)=I({\rm cl}(\overline{G}))^*=({\rm
Isol}(\overline{G}))^*+(I(\overline{G}))^*=(\{(t_1\cdots t_s)/t_i\mid t_i\in{\rm
Isol}(\overline{G})\})+I(\overline{G})^*.
$$
\quad Parts (c) and (d) follow from part (b).
\end{proof}

\begin{lemma}\label{jul14-22}
If $G$ is a graph and $U\subset V(\overline{G})$, then 
$$
\mbox{\rm(a) }\ \overline{G}\setminus{U}=\overline{G\setminus
U};\quad \mbox{\rm(b) }\ \overline{G}\setminus{\rm
Isol}(\overline{G})=\overline{G\setminus{\rm Isol}(\overline{G})}.
$$
\end{lemma}

\begin{proof} (a) To show equality we need to show that the vertex set
and edge set of the two graphs $\overline{G}\setminus{U}$ and $\overline{G\setminus
U}$ are equal. From the equalities
\begin{align*}
&V(\overline{G}\setminus U)=V(\overline{G})\setminus U=V(G)\setminus
{U},\\
&V(\overline{G\setminus U})=V(G\setminus U)=V(G)\setminus U,
\end{align*}
the vertex sets of the two graphs are equal. We set
$$H_0=\overline{G}\setminus U\, \mbox{  and }\, H=G\setminus U.
$$ 
\quad Note that $V(H_0)=V(\overline{H})=V(H)$. To show the inclusion $E(H_0)\subset
E(\overline{H})$ take
$e\in E(H_0)$. Then, $e\in E(\overline{G})$ and $e\cap U=\emptyset$. If $e\in E(H)$, then $e\in E(G)$ and
$e\cap U=\emptyset$, a contradiction. Thus, 
$e\in E(\overline{H})$. To show the inclusion
$E(H_0)\supset
E(\overline{H})$ take
$e\in E(\overline{H})$. Then, $e\notin E(H)$ and $e\subset
V(\overline{H})$. Hence, $e\in E(\overline{G})$ and $e\cap U=\emptyset$. Thus, $e\in E(H_0)$.

(b) Setting $U={\rm Isol}(\overline{G})$, this part follows from (a).
\end{proof}

\begin{lemma}\label{sc-lemma} Let $I$ be the edge ideal of a graph $G$, let $C_1$, 
$C_2$ be two odd cycles of $G$ with at most one common vertex, and let
$M\sb{C1,C_2}:=(\prod_{t_i\in C_1}t_i\prod_{t_i\in C_2}t_i)z\sp{{(|C_1|+|C_2|)}/{2}}$. The following hold.
\begin{enumerate}
\item[(a)] If $|C_1\cap C_2|=1$, then $M_{C_1,C_2}\in S[Iz]$.
\item[(b)] If $C_1\cap C_2=\emptyset$ and there is $e\in E(G)$
intersecting $C_1$ and $C_2$, then $M_{C_1,C_2}\in S[Iz]$.
\item[(c)] If $C_1,C_2$ form a Hochster configuration, then
$M_{C_1,C_2}\notin S[Iz]$.
\end{enumerate}
\end{lemma}

\begin{proof} We may assume that $C_1=\{t_1,\ldots,t_{\ell_1}\}$ and 
$C_2=\{t_{\ell+1},\ldots,t_{\ell_1+\ell_2}\}$ are odd cycles of lengths
$\ell_1$ and $\ell_2$, respectively.

(a) Assume that $C_1\cap C_2=\{t_{\ell_1}\}$ and
$t_{\ell_1}=t_{\ell_1+1}$. Then
\begin{align*}
M_{C_1,C_2}&=(t_1\cdots t_{\ell_1})(t_{\ell_1+1}\cdots
t_{\ell_1+\ell_2})z^{(\ell_1+\ell_2)/2}\\
&=(t_1\cdots t_{\ell_1-1} t_{\ell_1}^2z^{(\ell_1+1)/2})(t_{\ell_1+2}\cdots
t_{\ell_1+\ell_2}z^{(\ell_2-1)/2}).
\end{align*}
\quad Thus, $M_{C_1,C_2}\in I^{(\ell_1+1)/2}z^{(\ell_1+1)/2} 
I^{(\ell_2-1)/2}z^{(\ell_2-1)/2}=I^{(\ell_1+\ell_2)/2}z^{(\ell_1+\ell_2)/2}\subset
S[Iz]$.

(b) Assume that $C_1\cap C_2=\emptyset$ and
$\{t_{\ell_1},t_{\ell_1+1}\}$ is an edge of $G$. Then
\begin{align*}
M_{C_1,C_2}&=(t_1\cdots t_{\ell_1})(t_{\ell_1+1}\cdots
t_{\ell_1+\ell_2})z^{(\ell_1+\ell_2)/2}\\
&=(t_1\cdots t_{\ell_1}t_{\ell_1+1}z^{(\ell_1+1)/2})(t_{\ell_1+2}\cdots
t_{\ell_1+\ell_2}z^{(\ell_2-1)/2}).
\end{align*}
\quad Thus, $M_{C_1,C_2}\in I^{(\ell_1+1)/2}z^{(\ell_1+1)/2} 
I^{(\ell_2-1)/2}z^{(\ell_2-1)/2}=I^{(\ell_1+\ell_2)/2}z^{(\ell_1+\ell_2)/2}\subset
S[Iz]$.

(c) We argue by contradiction assuming that $M_{C_1,C_2}\in S[Iz]$. 
Then, $M_{C_1,C_2}\in I^mz^m$ for some $m\geq 1$, that is,
$\prod_{t_i\in C_1\cup C_2}t_i\in I^m$ and $m=(\ell_1+\ell_2)/2$. 
Thus
$$
(t_1\cdots t_{\ell_1})(t_{\ell_1+1}\cdots
t_{\ell_1+\ell_2})=t_{e_1}\cdots t_{e_m}t^\delta
$$
for some edges $e_1,\ldots,e_m$ of $G$. Hence, for each $1\leq i\leq
m$, either $t_{e_i}$ divides $t_1\cdots t_{\ell_1}$  or $t_{e_i}$
divides $t_{\ell_1+1}\cdots
t_{\ell_1+\ell_2}$. Thus, we may assume that $t_{e_1}\cdots t_{e_r}$
divides $t_1\cdots t_{\ell_1}$, $r\geq 1$, and $t_{e_{r+1}}\cdots
t_{e_m}$ divides $t_{\ell_1+1}\cdots
t_{\ell_1+\ell_2}$. Therefore, $\ell_1\geq 2r$ and $\ell_2\geq
2(m-r)$. As $\ell_1$ is odd, one has $\ell_1>2r$, and consequently 
$m=(\ell_1+\ell_2)/2>r+(m-r)=m$, a contradiction.
\end{proof}

\begin{theorem}{\rm(\cite[Conjecture~6.9]{ITG}, 
\cite[Corollary~5.8.10]{graphs})}\label{apr16-02} 
The edge ideal $I(G)$ of a 
graph $G$ is
normal if and only if $G$ admits no Hochster configurations.
\end{theorem}

\begin{proof} To show this result we use the following description of the
integral closure of the Rees algebra $S[Iz]$ of the ideal $I=I(G)$
\cite[Proposition~5.8.13]{graphs}: 
$$
\overline{S[Iz]}=S[Iz][\mathcal{B}'],
$$
where $\mathcal{B}'$ is the set of all 
monomials $M\sb{C1,C_2}:=(\prod_{t_i\in C_1}t_i\prod_{t_i\in
C_2}t_i)z\sp{{(|C_1|+|C_2|)}/{2}}$ 
such
that $C_1$ and $C_2$ are two induced odd cycles of $G$ with at most one common
vertex. If $C_1$ and
$C_2$ intersect at a point or $C_1$ and $C_2$ are joined by at least
one edge of $G$, then $M_{C_1,C_2}$ is in
$S[It]$ by Lemma~\ref{sc-lemma}. Hence, if $\mathcal{U}$ is the set of all 
monomials $M_{C_1,C_2}$ such that $C_1$, $C_2$ is a Hochster
configuration of $G$, one has the equality
$$
\overline{S[Iz]}=S[Iz][\mathcal{U}].
$$
\quad Therefore, by Lemma~\ref{sc-lemma}(c), $S[Iz]$ is normal if and only if $G$ has no Hochster
configurations, and the result follows from the fact that $I$ is
normal if and only if $S[Iz]$ is normal. 
\end{proof}

\begin{corollary}\label{normality-nc0}
Let $G$ be a graph. The following hold.
\begin{enumerate}
\item[(a)] If $\overline{G}$ is the disjoint union of two odd cycles of
length at least $5$, then $I_c(G)$ is not normal. 
\item[(b)] If $\overline{G}$ is an odd cycle of
length at least $5$, then $I_c(G)$ is normal. 
\end{enumerate}
\end{corollary}

\begin{proof} (a) As $\overline{G}$ has no triangles and no isolated
vertices, by
Lemma~\ref{jul11-22}, 
one has $I_c(G)=I(\overline{G})^*$. Then, 
by Proposition~\ref{duality-i-i*}, $I_c(G)$ is not normal if and only if
$I(\overline{G})$ is not normal. As $\overline{G}$ is a Hochster
configuration, by Theorem~\ref{apr16-02}, $I(\overline{G})$ is not
normal. Thus, $I_c(G)$ is not normal.

(b) As $\overline{G}$ has no triangles and no isolated vertices, by
Lemma~\ref{jul11-22}, 
one has the equality $I_c(G)=I(\overline{G})^*$. Then, 
by Proposition~\ref{duality-i-i*}, $I_c(G)$ is normal if and only if
$I(\overline{G})$ is normal. As $\overline{G}$ is an odd cycle, by
Theorem~\ref{apr16-02}, $I(\overline{G})$ is normal. Thus, $I_c(G)$ is
normal.
\end{proof}

\begin{corollary}\label{normality-nc}
Let $G$ be a graph. If $I_c(G)$ is normal, then $\overline{G}$ has no
Hochster configuration with induced odd cycles $C_1$, $C_2$ of
length at least $5$.  
\end{corollary}

\begin{proof} We argue by contradiction assuming that $\overline{G}$ 
has a Hochster configuration with induced odd cycles $C_1$, $C_2$ of
length at least $5$. Let $U$ be the set of vertices of $\overline{G}$
not in $V(C_1)\cup V(C_2)$. Then, by Proposition~\ref{may21-22},
$I_c(G\setminus U)$ is normal. The subgraph $\overline{G}\setminus U$
is the union $C_1\cup C_2$ of the cycles $C_1$ and $C_2$ because
$C_1\cup C_2$ is an induced subgraph of $\overline{G}$. Hence, by
Lemma~\ref{jul14-22}, $\overline{G\setminus U}=C_1\cup C_2$.
Therefore, by Corollary~\ref{normality-nc0}, $I_c(G\setminus U)$ is
not normal, a contradiction.
\end{proof}

\begin{theorem}[Duality criterion]\label{dim2-comb-char} Let $G$ be a 
graph with $\beta_0(G)\leq 2$. The 
following hold.
\begin{enumerate}
\item[(a)] $I_c(G)$ is normal if and only if $I(\overline{G})$ is normal. 
\item[(b)] $I_c(G)$ is normal if and only if $\overline{G}$ has no
Hochster configurations.
\end{enumerate}
\end{theorem}

\begin{proof} (a) $\Rightarrow$) Assume that $I_c(G)$ is normal. We
proceed by induction on $s=|V(G)|$. If $s=1$, then $I_c(G)=S$ and
$I(\overline{G})=(0)$, and if $s=2$, then either $G$ is a discrete
graph with two vertices, $I_c(G)=S$ and 
$I(\overline{G})=(t_1t_2)$, or $G=\mathcal{K}_2$, $I_c(G)=(t_1,t_2)$ and 
$I(\overline{G})=(0)$. Thus, $I(\overline{G})$ is normal in these
cases. Assume that $s\geq 3$. 

Case (I) $t_i$ is an isolated vertex of $\overline{G}$ for
some $i$. By Proposition~\ref{may21-22}, 
$I_c(G\setminus t_i)$ is normal. Then, by induction, 
$I(\overline{G\setminus t_i})$ is normal. As $\overline{G\setminus
t_i}=\overline{G}\setminus t_i$ and $t_i$ is isolated in  
$\overline{G}$, one has $I(\overline{G}\setminus
t_i)=I(\overline{G})$. Thus, $I(\overline{G})$ is normal.

Case (II) $\overline{G}$ has no isolated vertices. The graph
$\overline{G}$ has no triangles because $\beta_0(G)\leq 2$. Then, by
Lemma~\ref{jul11-22}, one has $I_c(G)=I(\overline{G})^*$. Hence,
the ideal $I(\overline{G})^*$ is normal and, by Proposition~\ref{duality-i-i*},
the ideal $I(\overline{G})$ is normal.

$\Leftarrow$) Assume that $I(\overline{G})$ is normal. By
Lemmas~\ref{jul11-22} and \ref{jul14-22}, one has  
$$
I_c(G\setminus {\rm Isol}(\overline{G}))=I\left(\overline{G\setminus{\rm
Isol}(\overline{G})}\right)^*=I\left(\overline{G}\setminus{\rm
Isol}(\overline{G})\right)^*=I\left(\overline{G}\right)^*.
$$
\quad Hence, the ideal $I_c(G\setminus {\rm Isol}(\overline{G}))$ is normal 
because $I\left(\overline{G}\right)^*$ is normal by
Proposition~\ref{duality-i-i*}. We set $H=G\setminus{\rm
Isol}(\overline{G})$. We may assume that $t_1,\ldots,t_r$ are the isolated vertices
of $\overline{G}$, then
$$
G=H\cup H_1\cup\cdots\cup H_r,
$$
where $H_i$ is the subgraph of $G$ given by 
\begin{align*}
&V(H_1)=V(H)\cup\{t_1\}\,\mbox{ and }\,E(H_1)=\{\{t_1,t_j\}\mid t_j\in
V(H)\}\,\mbox{ if }\,i=1,\\
&V(H_i)=V(H)\cup\{t_1,\ldots,t_i\}\,\mbox{ and }\,E(H_i)=\{\{t_i,t_j\}\mid t_j\in
V(H)\cup\{t_1,\ldots,t_{i-1}\}\,\mbox{ if }\,i\geq 2.
\end{align*}
\quad Setting $G_0=H$ and $G_i=H\cup H_1\cup\cdots\cup H_i$ for
$i=1,\ldots,r$, note that $G_i$ is the cone over $G_{i-1}$ with apex
$t_i$ for $i=1,\ldots,r$, i.e.,  $G_i$ is obtained from $G_{i-1}$ by 
joining every vertex of $G_{i-1}$ to $t_i$. As $I_c(H)$ is normal, by successively
applying Corollary~\ref{aug19-22}, we obtain that $I_c(G)$ is
normal.

(b) This part follows from part (a) and Theorem~\ref{apr16-02}.
\end{proof}

\section{Examples}\label{examples-section}
In this section we give some examples that complement our results. 
In particular, we show that Theorem~\ref{normal-b} does 
not extend to ideals generated by monomials of degree greater than
$2$ (Example~\ref{example1}). In
Examples~\ref{example3}--\ref{example5}, we illustrate
Lemma~\ref{jul11-22} and the duality criterion given in
Theorem~\ref{dim2-comb-char}. 
Then, we show that 
none of the implications of the duality criterion hold for arbitrary graphs
(Example~\ref{kaiser-example6}).

\begin{example}\label{example1} Let $S=\mathbb{Q}[t_1,\ldots,t_{10}]$ be a polynomial ring and
let $I=(t^{v_1},\ldots,t^{v_{10}})$ be the monomial ideal of $S$ generated by the set 
\begin{align*}
\mathcal{G}(I)=&\{t_1t_2t_3t_4t_5t_6t_7,\,t_1t_2t_3t_4t_5t_7t_8,\,t_1t_2t_3t_4t_5t_8t_9,
\,t_1t_2t_3t_4t_5t_8t_{10},\,t_1t_2t_3t_4t_7t_8t_{10},\,\\
&t_2t_3t_5t_7t_8t_9t_{10},\,t_1t_2t_6t_7t_8t_9t_{10},\,t_2t_3t_6t_7t_8t_9t_{10},
\,t_3t_4t_6t_7t_8t_9t_{10},\,t_3t_5t_6t_7t_8t_9t_{10}\}.
\end{align*}
\quad Then, using Procedures~\ref{procedure1} and
\ref{procedure1-bis}, we get that 
$\mathcal{B}=\{e_{11}\}\cup\{e_i+e_{11}\}_{i=1}^{10}\cup\{(v_i,1)\}_{i=1}^{10}$
is not a Hilbert basis and $I$ is a
normal ideal. Thus, Theorem~\ref{normal-b} does 
not extend to ideals generated by monomials of degree greater than
$2$.
\end{example}

\begin{example}\label{example2}
Let $S=\mathbb{Q}[t_1,\ldots,t_7]$ be a polynomial ring and let
$$
I=I(G)=(t_1t_3,t_1t_4,t_2t_4,t_1t_5,t_2t_5,t_3t_5,t_1t_6,
t_2t_6,t_3t_6,t_4t_6,t_2t_7,t_3t_7,t_4t_7,t_5t_7)
$$
be the edge ideal of the graph $G$ of Figure~\ref{figure1}. 
The complement $\overline{G}$ of $G$ is a cycle of length $7$. The graph $G$ is called an  
\textit{odd antihole} in the theory of perfect graphs \cite[p.~
71]{golumbic}.
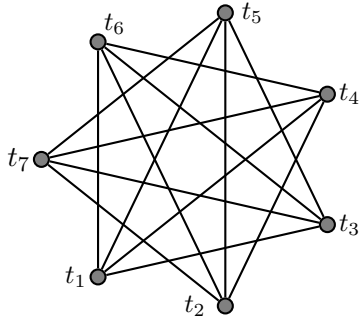
\begin{figure}[ht]
\setlength{\unitlength}{.035cm}
\begin{tikzpicture}[scale=1.0,thick]
\tikzstyle{every node}=[minimum width=0pt, inner sep=2pt]
\draw (-2,0) node (0) [draw, circle, fill=gray] {};
\draw (-2.3,0) node () {$t_7$};
\draw (-1.2469796037174672,-1.56366296493606) node (1) [draw, circle, fill=gray] {};
                        \draw (-1.54,-1.56366296493606) node () {\small $t_1$};
\draw (0.44504186791262895,-1.9498558243636472) node (2) [draw, circle, fill=gray] {};
\draw (0.042504186791262895,-1.9498558243636472) node () {\small $t_2$};			
\draw (1.8019377358048383,-0.8677674782351165) node (3) [draw, circle, fill=gray] {};
\draw (2.099377358048383,-0.8677674782351165) node () { \small $t_3$};
\draw (1.8019377358048383,0.8677674782351158) node (4) [draw, circle, fill=gray] {};
\draw (2.099377358048383,0.8677674782351158) node () { \small $t_4$};
\draw (0.44504186791262895,1.9498558243636472) node (5) [draw, circle, fill=gray] {};
\draw (0.80504186791262895,1.9498558243636472) node () {\small $t_5$};
\draw (-1.2469796037174665,1.5636629649360598) node (6) [draw, circle, fill=gray] {};
\draw (-1.02469796037174665,1.7996629649360598) node () { \small $t_6$};
\draw  (0) edge (2);
\draw  (0) edge (3);
\draw  (0) edge (4);
\draw  (0) edge (5);
\draw  (1) edge (3);
\draw  (1) edge (4);
\draw  (1) edge (5);
\draw  (1) edge (6);
\draw  (2) edge (4);
\draw  (2) edge (5);
\draw  (2) edge (6);
\draw  (3) edge (5);
\draw  (3) edge (6);
\draw  (4) edge (6);
\end{tikzpicture}
\caption{Graph $G$ is an odd antihole with $7$ vertices.}\label{figure1}
\end{figure}

\quad Using Procedure~\ref{normality-test-procedure} for 
\textit{Macaulay}$2$ \cite{mac2}, we obtain the
following information. 
The incidence matrix $B$ of $I_c(G)$ is
$$B=\left[\begin{matrix}
1&0&1&1&1&1&0\cr
1&1&1&1&1&0&0\cr
1&1&1&1&0&0&1\cr
1&1&1&0&0&1&1\cr
1&1&0&0&1&1&1\cr
0&1&0&1&1&1&1\cr
0&0&1&1&1&1&1
\end{matrix}
\right].
$$
\quad The ideals $I_c(G)$ and $I(G)$ are normal, and so are the ideals
$I_c(\overline{G})$ and $I(\overline{G})$. The normality of
$I(G)$, $I(\overline{G})$, and $I_c(G)$ also follow from
Theorem~\ref{apr16-02} and
Corollary~\ref{normality-nc0}, 
and the normality of $I_c(\overline{G})$ also
follows from \cite[Theorem~1.10]{Al-Ayyoub-Nasernejad-cover-ideals}. 
\end{example}

\begin{example}\label{two-triangles-example}
Let $G$ be a graph whose connected components are two triangles 
$G_1$ and $G_2$. Then, by Theorem~\ref{apr16-02}, $I(G_i)$ is normal
for $i=1,2$ but $I(G)$ is not normal by the same theorem.
\end{example}

\begin{example}\label{example3} If $G$ is the bipartite graph $\mathcal{K}_{1,2}$ with
edges $\{t_1,t_2\}$ and $\{t_1,t_3\}$. Then, $t_1$ is an isolated
vertex of $\overline{G}$, $E(\overline{G})=\{\{t_2,t_3\}\}$ and, by
Lemma~\ref{jul11-22}, one has
$$
I_c(G)=(t_2t_3)+I(\overline{G})^*=(t_2t_3)+(t_2t_3)^*=(t_2t_3,\,
t_1).
$$
\end{example}

\begin{example}\label{example4} If $G$ is a graph whose independence number
$\beta_0(G)$ is $1$, then $G=\mathcal{K}_s$ is a
complete graph, $I(\overline{G})=(0)$, $I_c(G)$ is normal
(Theorem~\ref{dim2-comb-char}) and, by Lemma~\ref{jul11-22}, 
one has
$$
I_c(G)=(\{(t_1\cdots t_s)/t_i\mid 1\leq i \leq s\}).
$$
\quad The normality of $I_c(G)$ also follows from
the fact that the ideal generated by all squarefree monomials of
$S$ of fixed degree $k\geq 1$ is normal \cite[Proposition~2.9]{Vi4}.
\end{example}

\begin{example}\label{example5} Let $G$ be the cone with apex $t_6$ over the cycle
$C_5=\{t_1,\ldots,t_5\}$. Then, $\beta_0(G)=2$, $I_c(G)$ is normal
(Theorem~\ref{dim2-comb-char}) and, by Lemma~\ref{jul11-22}, 
one has
$$
I_c(G)=(t_1t_2t_3t_4t_5,\,
t_2t_4t_5t_6,\,t_1t_2t_4t_6,\,t_1t_3t_4t_6,\,t_1t_3t_5t_6,\,t_2t_3t_5t_6).
$$
\end{example}

\begin{example}\cite[Fig.~1,
p.~241]{persistence-ce}\label{kaiser-example6} 
Let $I_c(G)$ be the ideal of covers of the graph $G$ defined by the 
generators of the following ideal 
\begin{eqnarray*}
&
I=&(t_1t_2,t_2t_3,t_3t_4,t_4t_5,t_5t_6,t_6t_7,t_7t_8,t_8t_9,t_9t_{10},t_1t_{10},\\
& &t_2t_{11},t_8t_{11},
t_3t_{12},t_7t_{12},t_1t_9,t_2t_8,t_3t_7,t_4t_6,t_1t_6,t_4t_9,\\ 
& &t_5t_{10},t_{10}t_{11},t_{11}t_{12},t_5t_{12}).
\end{eqnarray*}
\quad The graph $G$ is denoted by $H_4$ in \cite{persistence-ce}. Using the normality test of
Procedure~\ref{normality-test-procedure} and {\it Macaulay\/}$2$ \cite{mac2}, we
get that $I(G)$ and $I_c(G)$ are not normal whereas $I(\overline{G})$
and $I_c(\overline{G})$ are normal. This example shows that none of
the implications of the duality criterion of 
Theorem~\ref{dim2-comb-char}(a) hold for graphs with independence 
number $\beta_0(G)$ greater than $2$ because $\beta_0(G)=4$ and
$\beta_0(\overline{G})=3$. 
\end{example}

\begin{example}\label{example7} Let $G$ be the graph whose complement $\overline{G}$
is the graph depicted in Figure~\ref{figure2}. The graph $G$ has $50$
edges and $\beta_0(G)=3$. Using the normality
test of Procedure~\ref{normality-test-procedure} for 
\textit{Macaulay}$2$ \cite{mac2}, we obtain that $I(G)$ is normal, $I_c(G)$ is not 
normal, and furthermore $I_c(G)^5$ is not integrally
closed because one has
$$
f=t_1^4\,t_2^4\,t_3^4\,t_4^4\,t_5^4\,t_6^4\,t_7^2\,
t_8^4\,t_9^4\,t_{10}^4\,t_{11}^4\,t_{12}^4\,t_{13}^4\in\overline{I_c(G)^5}\setminus
I_c(G)^5.
$$
\quad This example shows that the Hochster configurations of
$\overline{G}$, with $C_1$,
$C_2$ cycles of length at least five, are not the only obstructions
for the normality of $I_c(G)$ (see Corollary~\ref{normality-nc}).  

\begin{figure}[ht]
\setlength{\unitlength}{.035cm}
\begin{tikzpicture}[scale=1.0,thick]
		\tikzstyle{every node}=[minimum width=0pt, inner sep=2pt, circle]
			\draw (1.44,-1.52) node[draw] (1) { \tiny 1};
			\draw (3.09,-1.81) node[draw] (2) { \tiny 2};
			\draw (4.41,-1.02) node[draw] (3) { \tiny 3};
			\draw (4.38,0.59) node[draw] (4) { \tiny 4};
			\draw (3,1.59) node[draw] (5) { \tiny 5};
			\draw (1.36,1.5) node[draw] (6) { \tiny 6};
			\draw (-2.68,0.7) node[draw] (7) { \tiny 10};
			\draw (-2.49,-1.06) node[draw] (8) { \tiny 11};
			\draw (-1.18,-1.89) node[draw] (9) { \tiny 12};
			\draw (0.06,-1.45) node[draw] (10) { \tiny 13};
			\draw (0.71,0.27) node[draw] (11) { \tiny 7};
			\draw (0.12,1.45) node[draw] (12) { \tiny 8};
			\draw (-1.42,1.69) node[draw] (13) { \tiny 9};
			\draw  (1) edge (3);
			\draw  (1) edge (4);
			\draw  (1) edge (5);
			\draw  (1) edge (6);
			\draw  (2) edge (4);
			\draw  (2) edge (5);
			\draw  (2) edge (6);
			\draw  (3) edge (5);
			\draw  (3) edge (6);
			\draw  (4) edge (6);
			\draw  (7) edge (9);
			\draw  (7) edge (10);
			\draw  (7) edge (11);
			\draw  (7) edge (12);
			\draw  (8) edge (10);
			\draw  (8) edge (11);
			\draw  (8) edge (12);
			\draw  (8) edge (13);
			\draw  (9) edge (11);
			\draw  (9) edge (12);
			\draw  (9) edge (13);
			\draw  (10) edge (12);
			\draw  (10) edge (13);
			\draw  (11) edge (13);
			\draw  (5) edge (11);
			\draw  (4) edge (11);
			\draw  (3) edge (11);
			\draw  (2) edge (11);
\end{tikzpicture}
\caption{Graph $\overline{G}$ consists of two antiholes joined by a vertex.}\label{figure2}
\end{figure}
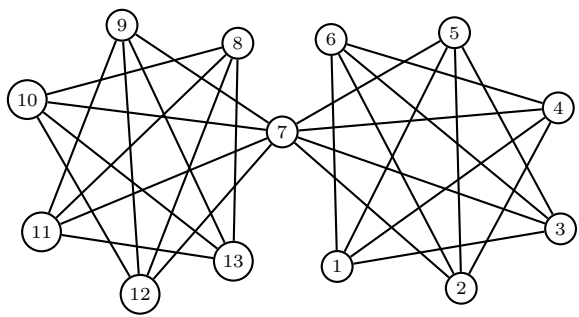
\end{example}
\begin{appendix}

\section{Procedures}\label{Appendix}

In this appendix we give procedures for \textit{Normaliz}
\cite{normaliz2} and \textit{Macaulay}$2$
\cite{mac2} to determine the normality of a monomial ideal, 
the minimal generators of the ideal of covers of a clutter, 
the Hilbert basis of the Rees cone $\mathbb{R}_+\mathcal{A}'$ defined in
Eq.~\eqref{rees-cone-eq}, and 
the Hilbert basis of the cone $\mathbb{R}_+\mathcal{B}$ generated by
the set $\mathcal{B}$ defined in
Eq.~\eqref{jul20-22}. The sets
$\mathcal{A}'$ and $\mathcal{B}$ are used to characterize the
normality of a monomial ideal (Lemma~\ref{normal-hilbert}) and the
normality of an ideal generated by monomials of degree two (Theorem~\ref{normal-b}).

\begin{procedure}\label{procedure1}
Let $I=(t^{v_1},\ldots,t^{v_q})$ be a monomial ideal of $S$. 
The following procedure for \textit{Normaliz} \cite{normaliz2}
computes the Hilbert basis 
of the cone generated by 
$$\mathcal{B}=\{e_{s+1}\}\cup\{e_i+e_{s+1}\}_{i=1}^s\cup\{(v_i,1)\}_{i=1}^q,$$
and determines whether or not $\mathcal{B}$ is a Hilbert basis.
The input is the matrix whose rows are 
the vectors in the set $\mathcal{B}$. This procedure corresponds to Example~\ref{example1}. 
\begin{verbatim}
amb_space 11
normalization 21
0 0 0 0 0 0 0 0 0 0 1
1 0 0 0 0 0 0 0 0 0 1
0 1 0 0 0 0 0 0 0 0 1
0 0 1 0 0 0 0 0 0 0 1
0 0 0 1 0 0 0 0 0 0 1
0 0 0 0 1 0 0 0 0 0 1
0 0 0 0 0 1 0 0 0 0 1
0 0 0 0 0 0 1 0 0 0 1
0 0 0 0 0 0 0 1 0 0 1
0 0 0 0 0 0 0 0 1 0 1
0 0 0 0 0 0 0 0 0 1 1
0 0 1 1 0 1 1 1 1 1 1
0 0 1 0 1 1 1 1 1 1 1
0 1 1 0 0 1 1 1 1 1 1
1 1 0 0 0 1 1 1 1 1 1
0 1 1 0 1 0 1 1 1 1 1
1 1 1 1 1 0 0 1 1 0 1
1 1 1 1 1 0 0 1 0 1 1
1 1 1 1 1 0 1 1 0 0 1
1 1 1 1 1 1 1 0 0 0 1
1 1 1 1 0 0 1 1 0 1 1
\end{verbatim}
\end{procedure}

\begin{procedure}\label{procedure1-bis}
Let $I=(t^{v_1},\ldots,t^{v_q})$ be a monomial ideal of $S$. 
The following procedure for \textit{Normaliz} \cite{normaliz2}
computes the Hilbert basis 
of the Rees cone of $I$ defined in Eq.~\eqref{rees-cone-eq} and
determines whether or not the set
$\mathcal{A}'=\{e_i\}_{i=1}^s\cup\{(v_i,1)\}_{i=1}^q$ is a Hilbert
basis. In
particular, by Lemma~\ref{normal-hilbert}, we can determine whether or
not $I$ is a normal ideal.  
The input is the matrix with rows $v_1,\ldots,v_q$.  
This procedure corresponds to Example~\ref{example1}. 
\begin{verbatim}
amb_space 11
rees_algebra 10
0 0 1 1 0 1 1 1 1 1 
0 0 1 0 1 1 1 1 1 1 
0 1 1 0 0 1 1 1 1 1 
1 1 0 0 0 1 1 1 1 1 
0 1 1 0 1 0 1 1 1 1 
1 1 1 1 1 0 0 1 1 0 
1 1 1 1 1 0 0 1 0 1 
1 1 1 1 1 0 1 1 0 0 
1 1 1 1 1 1 1 0 0 0 
1 1 1 1 0 0 1 1 0 1 
\end{verbatim}
\end{procedure}

\begin{procedure}[Normality test]\label{normality-test-procedure}
Let $I$ be a monomial ideal. We implement a procedure---that
uses the interface of \textit{Macaulay}$2$ \cite{mac2} 
to \textit{Normaliz} \cite{normaliz2}---to
determine the normality of $I$, and the normality and minimal
generators of the ideal of
covers of a clutter. This procedure corresponds to
Example~\ref{example2}. To compute other examples, in
the next procedure simply change the polynomial rings $R$ and $S$, and
the generators of $I$. 
\begin{verbatim}
restart
loadPackage("Normaliz",Reload=>true)
loadPackage("Polyhedra", Reload => true)
R=QQ[t1,t2,t3,t4,t5,t6,t7];
--antihole, complement of C7 
I=monomialIdeal(t1*t3,t1*t4,t1*t5,t1*t6,t2*t4,t2*t5,t2*t6,t2*t7,
t3*t5,t3*t6,t3*t7,t4*t6,t4*t7,t5*t7)
dim I
--Ideal of covers of clutter associated to I
J=dual(I)
--transpose incidence matrix of I
A = matrix flatten apply(flatten entries gens I , exponents)
--transpose incidence matrix of J
AJ=matrix flatten apply(flatten entries gens J , exponents)
--generators of Rees cone of I
M = id_(ZZ^(numcols(A)+1))^{0..numcols(A)-1}||
(A|transpose matrix {for i to numrows A-1 list 1})
--generators of Rees cone of J
MJ = id_(ZZ^(numcols(AJ)+1))^{0..numcols(AJ)-1}||
(AJ|transpose matrix {for i to numrows AJ-1 list 1})
--rows of M
l= entries M
--rows of MJ
lJ= entries MJ
--Next we compute the normalization of Rees algebras
S=QQ[t1,t2,t3,t4,t5,t6,t7,t8];
L=for i in l list S_i
LJ=for i in lJ list S_i
--Normalization of the Rees algebra of I
ICL=intclToricRing L
gens ICL
#gens ICL
flatten \ exponents \  gens ICL
--Normalization of the Rees algebra of J
ICLJ=intclToricRing LJ
gens ICLJ
flatten \ exponents \  gens ICLJ
--Normality test for ideal I
sort L==sort gens ICL
--Normality test for ideal J
sort LJ==sort gens ICLJ
\end{verbatim}
\end{procedure}

\end{appendix}
\section*{Acknowledgments} 
We thank the referee for a careful
reading of the paper. 
We used \textit{Normaliz} \cite{normaliz2} and \textit{Macaulay}$2$
\cite{mac2} to give a normality test for monomial ideals and to
compute Hilbert bases.  

\bibliographystyle{plain}

\begin{thebibliography}{10}

\bibitem{Al-Ayyoub-Nasernejad-cover-ideals}
I. Al-Ayyoub, M. Nasernejad and L. G. Roberts, 
Normality of cover ideals of graphs and normality under some
operations, Results Math. {\bf 74} (2019), 
no. 4, Paper No. 140, 26 pp. 

\bibitem{baum-trotter} S. Baum and L. E. Trotter, 
Integer rounding for polymatroid and branching optimization problems,
{\it SIAM J. Algebraic Discrete Methods} {\bf 2} (1981), no. 4, 416--425. 

\bibitem{ainv} J. P. Brennan, L. A. Dupont, and R. H. Villarreal, 
Duality, a-invariants and canonical modules         
of rings arising from linear optimization problems, 
{\it Bull. Math. Soc. Sci. Math. Roumanie} ({\it N.S.}) {\bf 51}  (2008),
no. 4, 279--305.

\bibitem{normaliz2} W. Bruns, B. Ichim, T. R\"omer, R. Sieg and C. S\"oger:
Normaliz. Algorithms for rational cones and affine monoids.
Available at \url{https://normaliz.uos.de}.

\bibitem{cornu-book}{G. Cornu\'ejols, {\it Combinatorial Optimization{\rm:} 
Packing and Covering\/}, CBMS-NSF Regional Conference Series in Applied 
Mathematics {\bf 74}, SIAM (2001).}

\bibitem{cox-toric-book} D. Cox, J. Little and H. Schenck, {\it
Toric Varieties}, Graduate Studies in Mathematics {\bf 124}, 
American Mathematical Society, Providence, RI, 2011. 

\bibitem{mc8} D. Delfino, A. Taylor, W. V. Vasconcelos, N. Weininger and
R. H. Villarreal, Monomial ideals and the computation of
multiplicities, \textit{Commutative ring theory and applications} 
(Fez, 2001), Lect. Notes Pure Appl. Math. {\bf 231}
(2003), 87--106,  
Dekker, New York, 2003. 

\bibitem{roundp} L. A. Dupont, C. Renter\'\i a and R. H. Villarreal, Systems with the
integer rounding property in normal monomial subrings, An. Acad.
Brasil. Ci$\hat{\rm e}$nc. {\bf 82} (2010), no. 4, 801--811. 

\bibitem{poset} L. A. Dupont and R. H. Villarreal, 
Edge ideals of clique clutters of comparability graphs and the
normality of monomial ideals, Math. Scand. {\bf 106} (2010), no. 1, 88--98.

\bibitem{ehrhart}{C. Escobar, J.
Mart\'\i nez-Bernal and R. H. Villarreal, Relative volumes and minors
in monomial subrings, {Linear Algebra Appl.} {\bf 374} (2003), 275--290.}

\bibitem{normali} C. Escobar, R. H. Villarreal and Y. Yoshino, Torsion
freeness and normality of blowup rings of monomial ideals, 
{\it Commutative Algebra\/}, Lect. Notes Pure Appl. Math. 
{\bf 244}, Chapman \& Hall/CRC, Boca Raton, FL, 2006, pp. 69--84. 

\bibitem{Fco-Ha-VT} C. A. Francisco, H. T. H$\rm\grave{a}$ and A. Van
Tuyl, A conjecture on critical graphs and 
connections to the persistence of associated primes, Discrete Math.
{\bf 310} (2010), 2176--2182.

\bibitem{Ful1}{W. Fulton, {\it Introduction to Toric Varieties\/}, 
Princeton University Press, 1993.}

\bibitem{gilmer}{R. Gilmer, {\it Commutative Semigroup Rings\/}, 
Chicago Lectures in Math., Univ. of Chicago 
Press, Chicago, 1984.}


\bibitem{reesclu}{I. Gitler, E. Reyes and R. H. Villarreal, 
Blowup algebras of square--free monomial ideals and some links to
combinatorial optimization problems, 
Rocky Mountain J. Math. {\bf 39} (2009), no. 1, 71--102.} 

\bibitem{graphs} I. Gitler and R. H. Villarreal, {\it Graphs, Rings and
Polyhedra\/}, Aportaciones Mat. Textos, {\bf 35}, Soc. Mat. Mexicana,
M\'exico, 2011.  

\bibitem{golumbic} M. C. Golumbic, {\it Algorithmic Graph Theory and
Perfect Graphs\/}, Second Edition, Annals of Discrete Mathematics  {\bf
57}, Elsevier Science B.V., Amsterdam, 2004. 


\bibitem{mac2} D. Grayson and M. Stillman, 
{\em Macaulay\/}$2$, 1996. Available at
\url{http://www.math.uiuc.edu/Macaulay2/}.

\bibitem{Ha-Trung-19} H. T. H$\rm\grave{a}$ and N.V. Trung, 
Membership criteria and containments of powers of monomial ideals.
Acta Math. Vietnam. {\bf 44} (2019), 117--139.

\bibitem{hemmecke} R. Hemmecke, On the computation of Hilbert bases
of cones,  Mathematical software (Beijing, 2002),  
307--317, World Sci. Publ., River Edge, NJ, 2002.

\bibitem{Herzog-Hibi-book} J. Herzog and T. Hibi, {\it Monomial
Ideals}, 
Graduate
Texts in  Mathematics {\bf 260}, Springer-Verlag, 2011.

\bibitem{huneke-swanson-book} 
C. Huneke and I. Swanson, {\it Integral Closure of Ideals Rings, and
Modules}, London Math. Soc., Lecture Note Series {\bf 336}, Cambridge
University Press, Cambridge, 2006.

\bibitem{persistence-ce} 
T. Kaiser, M. Stehl\'\i k and R. \v{S}krekovski, 
Replication in critical graphs and the persistence of monomial 
ideals, {J. Combin. Theory Ser. A} {\bf 123} (2014), no. 1,
239--251. 

\bibitem{ass-powers} J. Mart\'{\i}nez-Bernal, 
S. Morey and R. H.
Villarreal, Associated primes of powers of edge ideals, 
Collect. Math. {\bf 63} (2012),
no. 3, 361--374.

\bibitem{Schr1}{A. Schrijver, On total dual integrality, 
Linear Algebra Appl. {\bf 38} (1981), 27--32.}

\bibitem{Schr}{A. Schrijver, {\it Theory of Linear and Integer
Programming\/}, John Wiley \& Sons, New York, 1986.}

\bibitem{Schr2} {A. Schrijver, {\it Combinatorial Optimization\/}, 
Algorithms and Combinatorics {\bf 24}, Springer-Verlag, Berlin, 2003.}

\bibitem{ITG}{A. Simis, W.~V. Vasconcelos and R. H. Villarreal, On the
ideal theory of graphs, {J. Algebra} {\bf 167} (1994), 389--416.}

\bibitem{Vas1}{W. V. Vasconcelos, {\it Computational Methods in
Commutative Algebra and Algebraic Geometry\/}, 
Springer-Verlag, 1998.}

\bibitem{bookthree} W. V. Vasconcelos, {\it Integral Closure\/},
Springer Monographs in Mathematics, Springer-Verlag, New York, 2005.

\bibitem{Vi4}{R. H. Villarreal, Normality of subrings generated by
square free monomials, J. Pure Appl. Algebra {\bf 113} (1996),
91--106.}

\bibitem{perfect} R. H. Villarreal, Rees algebras and polyhedral cones of ideals 
of vertex covers of  perfect graphs, J. Algebraic Combin. {\bf 27}
(2008), 293--305.

\bibitem{monalg-rev} R. H. Villarreal, {\it Monomial Algebras\/},
Second edition, 
Monographs and Research Notes in Mathematics, Chapman and Hall/CRC,
Boca Raton, FL, 2015.

\end{thebibliography}

\section*{Statements and Declarations}  The authors declare that no funds, grants, or other support were
received during the preparation of this manuscript. 

The authors have
no relevant financial or non-financial interests to disclose. 

Data sharing not applicable to this article as no datasets were
generated or analysed during the current study. 

\end{document}